\documentclass[11pt,a4paper]{article}

\usepackage[T1]{fontenc}
\usepackage[utf8]{inputenc}
\usepackage{lmodern}
\usepackage{amsmath,amssymb,amsfonts,amsthm}
\usepackage{mathtools}
\usepackage{mathrsfs} 
\usepackage{graphicx}
\usepackage{booktabs}
\usepackage{xcolor}
\usepackage{microtype}
\usepackage[margin=1in]{geometry}
\usepackage[hidelinks]{hyperref}
\hypersetup{
  pdftitle={Ambient Chain-Relation Reduction in Contracting Controlled Skew-Product Semiflows},
  pdfauthor={Robert Vrabel},
  pdfkeywords={nonautonomous dynamics, skew-product semiflows, chain recurrence, contraction}
}

\theoremstyle{plain}
\newtheorem{theorem}{Theorem}
\newtheorem{lemma}[theorem]{Lemma}
\newtheorem{proposition}[theorem]{Proposition}
\newtheorem{corollary}[theorem]{Corollary}
\newtheorem{assumption}{Assumption}

\theoremstyle{remark}
\newtheorem{remark}{Remark}
\theoremstyle{definition}

\newtheorem{definition}{Definition}

\newenvironment{pf}{\begin{proof}}{\end{proof}}

\newcommand{\R}{\mathbb{R}}
\newcommand{\N}{\mathbb{N}}
\newcommand{\T}{\mathbb{T}}
\newcommand{\Mcal}{\mathcal{M}}

\newcommand{\Gcal}{\mathcal{G}}

\newcommand{\CR}{\operatorname{CR}}
\newcommand{\dist}{\operatorname{dist}}
\newcommand{\graph}{\operatorname{graph}}

\newcommand{\Lip}{\operatorname{Lip}}

\newcommand{\diam}{\operatorname{diam}}

\newcommand{\e}{\mathrm{e}}

\newcommand{\dd}{\mathrm{d}}  

\newcommand{\lambdaMax}{\lambda_{\max}}

\newcommand{\Core}{\operatorname{Core}}

\title{Ambient Chain-Relation Reduction in Contracting\\ Controlled Skew-Product Semiflows}
\author{Robert Vrabel\\[3pt]
\normalsize Institute of Applied Informatics, Automation, and Mechatronics\\
\normalsize Faculty of Materials Science and Technology\\
\normalsize Slovak University of Technology in Bratislava\\
\normalsize Jana Bottu 25, 917 24 Trnava, Slovakia\\[3pt]
\normalsize \texttt{robert.vrabel@stuba.sk}}
\date{}

\newcommand{\keywords}[1]{%
	\par\medskip
	\noindent\textbf{Keywords:} #1\par
}

\newcommand{\subjclass}[1]{%
	\par\smallskip
	\noindent\textbf{2020 Mathematics Subject Classification:} #1\par
}

\begin{document}
\maketitle 

\begin{abstract}
The qualitative dynamics of nonautonomous differential equations is often studied through cocycles and skew-product semiflows over a driving base. We develop a chain-recurrence reduction principle for systems admitting an equivariant attracting section and uniform asymptotic tracking. The abstract theorem determines the complete chain-attainability relation in the ambient phase space: chains may start off the attracting section, but every attainable target lies on it and is selected solely by the factor semiflow. Applied to contracting controlled skew-product semiflows, this gives exact descriptions of the chain recurrent set, chain components, Conley order, compact strictly invariant sets, attractors, basins, and dual repellers. A weighted-supremum transverse gauge yields a complete Lyapunov function for the full semiflow while preserving the recurrent values of the base. The differential hypotheses are verified through subadditive fiber-derivative growth and averaged moving-metric logarithmic norms, allowing transient expansion and non-unique ergodicity. Quantitative persistence is established for finite-window maps and smooth vector-field perturbations. The results provide a unified framework for chain recurrence, Lyapunov structure, and attractor--repeller geometry in contracting nonautonomous dynamics.
\end{abstract} 


\medskip

\section{Introduction}
\label{sec:introduction}

Skew-product semiflows provide a standard framework for the qualitative study of nonautonomous differential equations. A driving flow describes the evolution of an external signal, forcing, reference generator, or parameter process, while the state equation evolves in the corresponding fiber. Such representations also arise naturally in dynamical approaches to control systems \cite{ColoniusKliemann2000}. This representation separates recurrence carried by the base dynamics from recurrence generated internally by the state dynamics. Classical references on topological dynamics and nonautonomous equations include \cite{KloedenRasmussen2011,Sell1971,SellYou2002,ShenYi1998}; differential-equation-generated skew-product semiflows are treated, for example, in \cite{LongoNovoObaya2022,NovoObayaSanz2007}. The present paper is concerned with the global chain-recurrent structure of contracting skew-product extensions and with the way this structure is inherited from the driving flow.

The central topological mechanism is formulated in an abstract compact-semiflow setting. Let
$\Phi$ be a semiflow on a compact metric space $X$, let $\Psi$ be a semiflow on a compact metric space $Y$, and suppose that there are continuous maps
\begin{equation}
	p:X\to Y,
	\qquad
	i:Y\to X,
	\qquad
	p\circ i=\operatorname{id}_Y,
	\label{eq:abstract-factor-section-intro}
\end{equation}
that semiconjugate the dynamics and make $i$ an equivariant section. If every orbit of $\Phi$ uniformly approaches the orbit on the section with the same factor coordinate, then the ambient chain relation is completely determined by the factor system. Writing
\(
x\rightsquigarrow_\Phi y
\)
to mean that $y$ is chain attainable from $x$ under $\Phi$, the reduction takes the exact form
\begin{equation}
	x\rightsquigarrow_\Phi y
	\quad\Longleftrightarrow\quad
	y=i(\eta)\ \text{for some }\eta\in Y
	\text{ such that }p(x)\rightsquigarrow_\Psi\eta.
	\label{eq:abstract-chain-reduction-intro}
\end{equation}
This source--target identity is intrinsically asymmetric: the source may lie outside the equivariant section, but every chain target lies on it. It implies exact identities for chain-attainable sets, chain recurrence, chain components, and the order induced on those components.

We apply this abstract reduction to feedback-controlled nonautonomous differential equations of the form
\begin{equation}
 \dot x
 =f(\theta_t\omega,x)+B(\theta_t\omega,x)k(\theta_t\omega,x)
 =:F(\theta_t\omega,x)
 \label{eq:controlled-system-intro}
\end{equation}
over a compact base flow $\theta$. If the fiber dynamics contracts over one common finite time window, the graph-transform mechanism yields a unique continuous invariant graph
\begin{equation}
 \Gamma_h
 =\{(\omega,h(\omega)):\omega\in\Omega\}
 \label{eq:invariant-graph-intro}
\end{equation}
and uniform exponential tracking of the corresponding graph trajectory. Taking $p=\pi_\Omega$ and $i=\iota$, where
\[
\pi_\Omega(\omega,x)=\omega,
\qquad
\iota(\omega)=(\omega,h(\omega)),
\]
formula \eqref{eq:abstract-chain-reduction-intro} becomes
\begin{equation}
	z=(\omega,x)\rightsquigarrow_\Pi w
	\quad\Longleftrightarrow\quad
	w=\iota(\eta)\ \text{for some }\eta\in\Omega
	\text{ such that }\omega\rightsquigarrow_\theta\eta.
	\label{eq:ambient-chain-relation-intro}
\end{equation}
For a semiflow $\Phi$, denote the positive chain-attainable set of $x$ and the chain-recurrent set by
\[
\mathcal C_\Phi^+(x)
:=
\{y:x\rightsquigarrow_\Phi y\},
\qquad
\CR(\Phi)
:=
\{x:x\rightsquigarrow_\Phi x\},
\]
respectively. Hence \eqref{eq:ambient-chain-relation-intro} is equivalently written as
\begin{equation}
	\mathcal C_\Pi^+(z)
	=
	\iota\bigl(\mathcal C_\theta^+(\pi_\Omega z)\bigr).
	\label{eq:chain-attainable-intro}
\end{equation}
Consequently,
\begin{equation}
	\CR(\Pi)
	=
	\{(\omega,h(\omega)):\omega\in\CR(\theta)\},
	\label{eq:main-cr-identity-intro}
\end{equation}
and the chain components, together with their induced order, are transported without merging or splitting.

The invariant-set structure displays a related but distinct asymmetry. Compact strictly invariant subsets of the skew product lie on $\Gamma_h$, and attractors are graph lifts
\begin{equation}
 A\longmapsto\widehat A:=\iota(A).
 \label{eq:attractor-bijection-intro}
\end{equation}
Their basins satisfy the exact identity
\begin{equation}
 \mathcal B_\Pi(\widehat A)
 =\pi_\Omega^{-1}\bigl(\mathcal B_\theta(A)\bigr).
 \label{eq:basin-intro}
\end{equation}
Standard dual repellers do not collapse to the graph. If \(A_\theta^*\) is the dual repeller of \(A\), then
\begin{equation}
	(\widehat A)_\Pi^*
	=\pi_\Omega^{-1}(A_\theta^*),
	\qquad
	\Core_\Pi^{\mathrm{str}}\bigl((\widehat A)_\Pi^*\bigr)
	=\iota(A_\theta^*).
	\label{eq:dual-repeller-asymmetry-intro}
\end{equation}
Here
\(\Core_\Pi^{\mathrm{str}}\bigl((\widehat A)_\Pi^*\bigr)\)
denotes the largest compact strictly invariant subset of the dual repeller. Thus attractors and chain targets lie on the graph, whereas standard dual repellers retain full fibers and only their strictly invariant cores collapse to graph lifts.

The finite-window hypothesis is obtained from an asymptotic differential criterion. Let $\kappa_t(\omega)$ be the logarithm of the maximal fiber derivative norm accumulated during time $t$. The family $\{\kappa_t\}_{t\geq0}$ is subadditive. Negativity of its maximal exponent over all invariant probability measures of the base is converted into a common negative time window by the continuous-time semi-uniform theorem of Schreiber \cite{Schreiber1998}; related principles and forced-system applications were developed in \cite{SturmanStark2000}. This asymptotic-growth viewpoint is also compatible with classical dichotomy and spectral approaches to nonautonomous linear differential equations \cite{Coppel1978,SackerSell1978}. For differential equations, let \(\Mcal_\theta(\Omega)\) denote the set of all
\(\theta\)-invariant Borel probability measures on \(\Omega\). A moving quadratic metric then gives the verifiable sufficient condition
\begin{equation}
	\sup_{\mu\in\Mcal_\theta(\Omega)}
	\int_\Omega \nu_P(\omega)\,\dd\mu(\omega)<0,
	\label{eq:average-contraction-intro}
\end{equation}
where \(\nu_P\) is the associated moving-metric logarithmic-norm bound. This condition permits a positive instantaneous logarithmic norm on transient portions of base trajectories and does not require unique ergodicity.

\medskip\noindent\textbf{Main contributions and novelty.}
Uniform fiber contraction is known to generate attracting invariant graphs in forced systems
\cite{Stark1997,Stark1999}, while semi-uniform subadditive theorems convert negative
growth rates over invariant measures into common finite-time contraction estimates
\cite{Schreiber1998,SturmanStark2000}. Extensions with shrinking fibers preserve a broad
range of measure-theoretic and statistical properties, including invariant measures and
entropy \cite{Kloeckner2021}; recent work also addresses linear response for
contracting-fiber skew products \cite{AlvesBahsoun2026}. The Morse structure of skew-product semiflows and Lyapunov
functions associated with lifted finite Morse decompositions were developed in
\cite{BortolanEtAl2013}, whereas attractor lattices, dual repellers, and their order
representations provide the global topological framework used below
\cite{KaliesMischaikowVandervorst2022,KaliesVandervorst2025}.

The present paper connects these analytical and topological ingredients through an exact
ambient source--target classification. The essential point is that the initial point of a
chain may lie anywhere in the ambient phase space, including outside the attracting
section, whereas every chain-attainable target is forced onto that section. This reduction
identifies not only the recurrent dynamics carried by the invariant graph, but also the
directional chain structure generated by arbitrary transient fiber states. In the controlled
skew-product setting, it leads to a complete description of the recurrent core, its component
order, the compact strictly invariant dynamics, the attractor structure, and the complementary
full-fiber geometry of dual repellers. The main contributions are concentrated in the following results.

\begin{enumerate}
	
	\item \textit{Exact ambient chain reduction.}
	Theorem~\ref{thm:abstract-chain-reduction} establishes the source--target identity
	\[
	x\rightsquigarrow_\Phi y
	\quad\Longleftrightarrow\quad
	y=i(\eta)
	\ \text{for some }\eta\in Y
	\text{ such that }
	p(x)\rightsquigarrow_\Psi\eta .
	\]
	The result applies to arbitrary ambient source and target points. Its distinguishing
	content is the treatment of chains initiated away from the attracting section: uniform
	tracking removes the transverse position of the source from the asymptotic chain geometry,
	while the factor dynamics completely determines the attainable target. Thus the theorem
	goes strictly beyond conjugacy on the section and gives an explicit reduction of the
	ambient chain-attainability relation. Corollary~\ref{cor:abstract-recurrent-reduction}
	then identifies the chain-recurrent core and proves that the induced order of its chain
	components is preserved exactly.
	
	\item \textit{Complete chain classification of the controlled skew product.}
	Corollary~\ref{cor:controlled-chain-reduction} applies the abstract reduction theorem to
	the invariant graph of the controlled cocycle. Theorem~\ref{thm:chain-recurrence-identity}
	gives the exact identity
	\[
	\CR(\Pi)=\iota\bigl(\CR(\theta)\bigr),
	\]
	while Theorem~\ref{thm:chain-component-correspondence} proves that chain components and
	their directed order are transported without merging, splitting, or creating additional
	recurrent classes. Consequently, every fiber point has precisely the same collection of
	chain-attainable targets as its base point, lifted to the invariant graph. The result
	therefore classifies the ambient chain geometry, not merely the recurrent dynamics already
	confined to the graph.
	
	\item \textit{Constructive complete Lyapunov lift.}
	Theorem~\ref{thm:complete-lyapunov-lift} introduces the weighted-supremum transverse gauge
	\[
	W_\alpha(z)
	=
	\sup_{t\geq0}\e^{\alpha t}V(\Pi_tz),
	\]
	which converts asymptotic fiber tracking into a globally monotone quantity with the
	exponential semigroup estimate
	\[
	W_\alpha(\Pi_tz)
	\leq
	\e^{-\alpha t}W_\alpha(z).
	\]
	The function vanishes precisely on the invariant graph and decreases strictly along every
	off-graph trajectory. Combining it with a complete Lyapunov function of the base yields
	\[
	L_\Pi
	=
	L_\theta\circ\pi_\Omega+\gamma W_\alpha,
	\qquad \gamma>0.
	\]
	This construction acts at the level of the full chain-recurrent decomposition, preserves
	all recurrent values exactly, and provides strict transverse decay even when the original
	graph distance exhibits transient growth. It complements the general existence theory of
	complete Lyapunov functions \cite{Patrao2011} and the lifting of Lyapunov functions for
	finite Morse decompositions in \cite[Theorem~5.7]{BortolanEtAl2013} by giving an explicit
	chain-level formula adapted to the contracting skew-product geometry.
	
	\item \textit{A sharp attractor--repeller duality.}
	Theorem~\ref{thm:compact-invariant-correspondence} classifies all compact strictly
	invariant sets, and Theorem~\ref{thm:attractor-correspondence} gives exact correspondences
	of attractors and their basins. Theorem~\ref{thm:attractor-repeller-correspondence}
	reveals a particularly striking dual picture:
	\[
	(\iota(A))_\Pi^*
	=
	\pi_\Omega^{-1}(A_\theta^*),
	\qquad
	\Core_\Pi^{\mathrm{str}}
	\bigl((\iota(A))_\Pi^*\bigr)
	=
	\iota(A_\theta^*).
	\]
	Thus the attracting and repelling sides of the dynamics respond differently to fiber
	contraction. Chain targets, compact strictly invariant sets, and attractors collapse onto
	the invariant graph, whereas a standard dual repeller retains the entire fibers over the
	corresponding base repeller. Only its largest compact strictly invariant core collapses
	back to the graph. This sharp asymmetry is the attractor--repeller counterpart of the
	source--target asymmetry in the ambient chain-reduction theorem. The associated
	attractor--repeller pairs and the attractor lattice are consequently transported by explicit
	order isomorphisms.
	
	\item \textit{Differential realization and quantitative robustness.}
	Theorem~\ref{thm:subadditive-main} converts negativity of the maximal invariant-measure
	exponent of the fiber derivative cocycle into a common finite-time contraction window.
	This provides the analytical bridge from averaged infinitesimal information to the uniform
	tracking hypothesis required by the abstract topological theory.
	Theorem~\ref{thm:averaged-contraction-main} supplies a verifiable moving-metric
	logarithmic-norm certificate in a base-dependent quadratic metric, allowing both transient
	fiber expansion and non-unique ergodicity of the driving flow. Finally,
	Theorem~\ref{thm:graph-robustness}, Proposition~\ref{prop:C1-to-window-bounds}, and
	Corollary~\ref{cor:CR-robustness} establish quantitative persistence of the invariant graph,
	the contraction window, and the chain-recurrent set under finite-window and vector-field
	perturbations.
	
\end{enumerate}

The paper is organized as follows. Section~\ref{sec:setting} introduces the controlled cocycle, the moving fiber metric, and the chain relation. Section~\ref{sec:abstract-reduction} proves the abstract chain-reduction theorem. Section~\ref{sec:graph} constructs the invariant graph and establishes exponential tracking. Section~\ref{sec:recurrent} applies the abstract result, constructs complete Lyapunov functions, and derives the invariant-set and attractor--repeller correspondences. Sections~\ref{sec:subadditive-contraction} and \ref{sec:averaged-contraction} provide the subadditive and moving-metric contraction criteria. Quantitative robustness is treated in Section~\ref{sec:robustness}, and representative examples are given in Section~\ref{sec:examples}.

\section{Controlled skew-product semiflows and chain relations}
\label{sec:setting}

This section introduces the controlled skew-product framework and fixes the
dynamical and metric notation used throughout the paper. We first specify the
base flow and the feedback-controlled cocycle, and then introduce the moving
fiber metric and the chain-attainability relation.

\subsection{Base flow and controlled cocycle}
\label{subsec:base-cocycle}

Let \((\Omega,d_\Omega)\) be a compact metric space and let
\begin{equation}
 \theta:\R\times\Omega\to\Omega,
 \qquad (t,\omega)\mapsto\theta_t\omega,
 \label{eq:base-flow}
\end{equation}
be a continuous flow. The base is interpreted as an exosystem. It may be a periodic or quasiperiodic phase flow, a compact hull of a translation-compact forcing, or a more general compact dynamical system.

Let \(K\subset\R^n\) be nonempty, compact, and convex. The
feedback-controlled equation introduced in
\eqref{eq:controlled-system-intro} is considered on \(K\) in the form
\begin{equation}
	\dot x=F(\theta_t\omega,x),
	\qquad
	x(0)=x_0\in K,
	\label{eq:closed-loop-ode}
\end{equation}
where
\begin{equation}
	F(\omega,x)
	=
	f(\omega,x)+B(\omega,x)k(\omega,x).
	\label{eq:closed-loop-vector-field}
\end{equation}
The map \(F\) is assumed to extend continuously to \(\Omega\times U\), where \(U\subset\R^n\) is an open neighborhood of \(K\). Its fiber derivative \(D_xF\) exists on \(\Omega\times U\), is continuous there, and is locally bounded uniformly with respect to the base variable. These hypotheses imply local existence, uniqueness, continuous dependence, and differentiability of the fiber solution map with respect to its initial state. We impose the following standing global hypothesis.

\begin{assumption}
\label{ass:well-posedness}
For every \((\omega,x)\in\Omega\times K\), equation \eqref{eq:closed-loop-ode} has a unique solution defined for all \(t\geq0\), and \(K\) is positively invariant:
\begin{equation}
 \varphi(t,\omega,x)\in K,
 \qquad t\geq0.
 \label{eq:positive-invariance}
\end{equation}
The solution map \(\varphi:\R_+\times\Omega\times K\to K\) is continuous.
\end{assumption}

The solution map satisfies the cocycle identities
\begin{equation}
 \varphi(0,\omega,x)=x,
 \qquad
 \varphi(t+s,\omega,x)
 =\varphi\bigl(t,\theta_s\omega,\varphi(s,\omega,x)\bigr).
 \label{eq:cocycle-identity}
\end{equation}
The associated skew-product semiflow is
\begin{equation}
 \Pi_t(\omega,x)
 =\bigl(\theta_t\omega,\varphi(t,\omega,x)\bigr),
 \qquad t\geq0.
 \label{eq:skew-product-semiflow}
\end{equation}
It acts on the compact phase space
\begin{equation}
 Z:=\Omega\times K.
 \label{eq:phase-space-Z}
\end{equation}

\subsection{A continuous family of fiber metrics}
\label{subsec:fiber-metrics}

Let \(P:\Omega\to\R^{n\times n}\) be continuous, symmetric, and uniformly positive definite. Thus, there exist constants \(0<m_P\leq M_P<\infty\) such that
\begin{equation}
 m_PI\preceq P(\omega)\preceq M_PI,
 \qquad \omega\in\Omega.
 \label{eq:P-uniform-bounds}
\end{equation}
Define
\begin{equation}
 \lvert v\rvert_\omega
 :=\sqrt{v^\top P(\omega)v},
 \qquad v\in\R^n.
 \label{eq:moving-norm}
\end{equation}
The family \(\lvert\cdot\rvert_\omega\) is uniformly equivalent to the Euclidean norm. For the quantitative estimates below, we fix the product metric on \(Z\) by
\begin{equation}
 D\bigl((\omega,x),(\eta,y)\bigr)
 :=d_\Omega(\omega,\eta)+\lVert x-y\rVert.
 \label{eq:product-metric}
\end{equation}
For points with the same base coordinate,
\[
 D\bigl((\omega,x),(\omega,y)\bigr)
 =\lVert x-y\rVert
 \leq m_P^{-1/2}\lvert x-y\rvert_\omega.
\]
On the compact space \(Z\), the chain relation and chain recurrence are independent of the particular compatible metric; the Hausdorff estimates in this paper refer to the fixed metric \(D\) in \eqref{eq:product-metric}.

\begin{definition}
\label{def:block-contraction}
The cocycle \(\varphi\) is said to be uniformly contracting over the window \(\tau>0\) if there exists \(q\in(0,1)\) such that
\begin{equation}
 \left\lvert
 \varphi(\tau,\omega,x)-\varphi(\tau,\omega,y)
 \right\rvert_{\theta_\tau\omega}
 \leq q\lvert x-y\rvert_\omega
 \label{eq:block-contraction}
\end{equation}
for all \(\omega\in\Omega\) and \(x,y\in K\).
\end{definition}

Finite-window contraction is weaker than an instantaneous differential contraction condition. The dynamics may expand during subintervals of \([0,\tau]\), provided that the net effect over the complete window is contractive.

By compactness and continuous differentiability, there exists \(L_\tau\geq1\) such that
\begin{equation}
 \left\lvert
 \varphi(t,\omega,x)-\varphi(t,\omega,y)
 \right\rvert_{\theta_t\omega}
 \leq L_\tau\lvert x-y\rvert_\omega,
 \qquad 0\leq t\leq\tau.
 \label{eq:intermediate-lipschitz}
\end{equation}
This estimate will be used to pass from contraction at integer multiples of \(\tau\) to continuous-time exponential tracking.

\subsection{Chains, chain relations, and semiconjugacies}
\label{subsec:chains}

We use the Conley chain notion for continuous semiflows on compact metric spaces; see \cite{Conley1978,Hurley1995}. The lower bound on the flight times is essential for excluding artificial recurrence generated by arbitrarily short pieces of orbit.

\begin{definition}
\label{def:chain}
Let $\Phi$ be a continuous semiflow on a compact metric space $(X,d_X)$. For $\varepsilon>0$ and $T>0$, an $(\varepsilon,T)$-chain from $x_0\in X$ to $x_m\in X$ is a finite sequence with $m\geq1$
\begin{equation}
 x_0,x_1,\ldots,x_m\in X,
 \qquad
 t_1,\ldots,t_m\geq T,
 \label{eq:chain-sequence}
\end{equation}
such that
\begin{equation}
 d_X\bigl(\Phi_{t_i}x_{i-1},x_i\bigr)<\varepsilon,
 \qquad i=1,\ldots,m.
 \label{eq:chain-jumps}
\end{equation}
We write
\begin{equation}
 x\rightsquigarrow_\Phi y
 \label{eq:chain-relation-notation}
\end{equation}
if an $(\varepsilon,T)$-chain from $x$ to $y$ exists for every $\varepsilon,T>0$. It is transitive but, outside the chain recurrent set, need not be reflexive. A point $x$ is chain recurrent if $x\rightsquigarrow_\Phi x$, and
\begin{equation}
 \CR(\Phi):=\{x\in X:x\rightsquigarrow_\Phi x\}.
 \label{eq:chain-recurrent-set-definition}
\end{equation}
Mutual chain attainability defines an equivalence relation on $\CR(\Phi)$; its equivalence classes are the chain components. We use this notation for $\Pi$, $\theta$, and the abstract semiflows introduced below.
\end{definition}

The projection
\begin{equation}
 \pi_\Omega:Z\to\Omega,
 \qquad \pi_\Omega(\omega,x)=\omega,
 \label{eq:base-projection}
\end{equation}
semiconjugates \(\Pi\) and \(\theta\):
\begin{equation}
 \pi_\Omega\circ\Pi_t
 =\theta_t\circ\pi_\Omega,
 \qquad t\geq0.
 \label{eq:projection-semiconjugacy}
\end{equation}
The following elementary observation will be used repeatedly.

\begin{lemma}
\label{lem:chains-semiconjugacy}
Let \(\Phi\) and \(\Psi\) be continuous semiflows on compact metric spaces \(X\) and \(Y\), respectively, and let \(p:X\to Y\) be continuous with
\begin{equation}
 p\circ\Phi_t=\Psi_t\circ p,
 \qquad t\geq0.
 \label{eq:abstract-semiconjugacy}
\end{equation}
Then
\begin{equation}
 x\rightsquigarrow_\Phi x'
 \quad\Longrightarrow\quad
 p(x)\rightsquigarrow_\Psi p(x').
 \label{eq:chain-relation-semiconjugacy}
\end{equation}
If \(p\) is a homeomorphism from a compact invariant set \(X_0\subset X\) onto a compact invariant set \(Y_0\subset Y\), and conjugates the restricted semiflows, then the chain relations of the two restrictions are equivalent under \(p\).
\end{lemma}

\begin{pf}
Fix \(\varepsilon,T>0\). Uniform continuity of \(p\) gives \(\delta>0\) such that
\begin{equation}
 d_X(u,v)<\delta
 \quad\Longrightarrow\quad
 d_Y(p(u),p(v))<\varepsilon.
 \label{eq:semiconjugacy-modulus}
\end{equation}
Applying \(p\) to a \((\delta,T)\)-chain and using \eqref{eq:abstract-semiconjugacy} produces an \((\varepsilon,T)\)-chain in \(Y\). The conjugacy statement follows by applying the same argument to \(p\) and to \(p^{-1}\).
\end{pf}

In particular, \eqref{eq:projection-semiconjugacy} implies that every skew-product chain projects to a base chain after reducing the jump tolerance according to the modulus of continuity of \(\pi_\Omega\).

\section{Abstract chain reduction under an attracting equivariant section}
\label{sec:abstract-reduction}

We formulate the chain-reduction principle in an abstract compact-semiflow
setting. The relevant structure consists of a factor map, an equivariant
section, and uniform asymptotic tracking of that section by all ambient
trajectories. These hypotheses are sufficient to determine the full ambient
chain relation from the chain relation of the factor semiflow.

\begin{assumption}
\label{ass:abstract-tracking}
Let $\Phi$ and $\Psi$ be continuous semiflows on compact metric spaces $(X,d_X)$ and $(Y,d_Y)$, respectively. Assume that there are continuous maps
\begin{equation}
 p:X\to Y,
 \qquad
 i:Y\to X,
 \label{eq:abstract-factor-section}
\end{equation}
such that
\begin{equation}
 p\circ i=\operatorname{id}_Y,\qquad
 p\circ\Phi_t=\Psi_t\circ p,\qquad
 \Phi_t\circ i=i\circ\Psi_t,\qquad t\geq0.
 \label{eq:abstract-factor-equivariant-identities}
\end{equation}
Thus $p$ is a factor map and $i$ is an equivariant section of $p$. In addition, assume uniform asymptotic tracking of this section:
\begin{equation}
 r(T):=
 \sup_{t\geq T}\sup_{x\in X}
 d_X\bigl(\Phi_t x,i(\Psi_t p(x))\bigr)
 \longrightarrow0
 \qquad (T\to\infty).
 \label{eq:abstract-uniform-tracking}
\end{equation}
\end{assumption}

Because $p\circ i=\operatorname{id}_Y$, the map $i$ is injective. Since $Y$ is compact and $X$ is Hausdorff, $i$ is a homeomorphism from $Y$ onto the compact set $i(Y)$. Equivariance gives $\Phi_t(i(Y))=i(\Psi_t(Y))\subseteq i(Y)$, so $i(Y)$ is positively invariant. It is strictly invariant only when the factor semiflow is surjective.

\begin{theorem}
\label{thm:abstract-chain-reduction}
Under Assumption~\ref{ass:abstract-tracking}, for all $x,y\in X$,
\begin{equation}
 x\rightsquigarrow_\Phi y
 \quad\Longleftrightarrow\quad
 y=i(\eta)\ \text{for some }\eta\in Y
 \text{ with }p(x)\rightsquigarrow_\Psi\eta.
 \label{eq:abstract-chain-reduction}
\end{equation}
Equivalently,
\begin{equation}
 \mathcal C_\Phi^+(x)
 =i\bigl(\mathcal C_\Psi^+(p(x))\bigr),
 \qquad
 \mathcal C_\Phi^+(x):=\{y:x\rightsquigarrow_\Phi y\}.
 \label{eq:abstract-chain-attainable-identity}
\end{equation}
\end{theorem}

\begin{pf}
We prove the two implications separately.

Assume first that $x\rightsquigarrow_\Phi y$. The factor relation implies
\begin{equation}
 p(x)\rightsquigarrow_\Psi p(y).
 \label{eq:abstract-reduction-projection}
\end{equation}
Indeed, given $\varepsilon_Y,T>0$, uniform continuity of $p$ provides $\delta_X>0$ such that $d_X(u,v)<\delta_X$ implies $d_Y(p(u),p(v))<\varepsilon_Y$. Projecting a $(\delta_X,T)$-chain from $x$ to $y$ through $p$ and using $p\circ\Phi_t=\Psi_t\circ p$ yields an $(\varepsilon_Y,T)$-chain from $p(x)$ to $p(y)$.

It remains to identify the target. Fix $T>0$ and $\varepsilon>0$. Since $x\rightsquigarrow_\Phi y$, there is an $(\varepsilon,T)$-chain
\begin{equation}
 x=x_0,x_1,\ldots,x_m=y,\qquad t_1,\ldots,t_m\geq T,
 \label{eq:abstract-reduction-chain}
\end{equation}
with $m\geq1$. The final jump and the tracking hypothesis give
\begin{equation}
 \dist\bigl(y,i(Y)\bigr)
 \leq d_X\bigl(y,\Phi_{t_m}x_{m-1}\bigr)
      +\dist\bigl(\Phi_{t_m}x_{m-1},i(Y)\bigr)
 <\varepsilon+r(T).
 \label{eq:abstract-target-distance}
\end{equation}
For fixed $T$, letting $\varepsilon\downarrow0$ yields $\dist(y,i(Y))\leq r(T)$. The estimate holds for every $T>0$, and therefore $r(T)\to0$ implies $\dist(y,i(Y))=0$. The set $i(Y)$ is compact and hence closed, so $y=i(\eta)$ for some $\eta\in Y$. This point is unique because $i$ is injective. Applying $p$ and using $p\circ i=\operatorname{id}_Y$ gives $\eta=p(y)$. Notice that the estimate \eqref{eq:abstract-target-distance} uses only the final leg of the chain; no control of the number of intermediate legs is required. Combining the identity $\eta=p(y)$ with \eqref{eq:abstract-reduction-projection} proves the necessity.

Conversely, suppose that $p(x)\rightsquigarrow_\Psi\eta$ and set $y=i(\eta)$. Let $\varepsilon,T>0$ be arbitrary. Uniform continuity of $i$ gives $\delta>0$ such that
\begin{equation}
 d_Y(u,v)<\delta
 \quad\Longrightarrow\quad
 d_X(i(u),i(v))<\frac{\varepsilon}{2}.
 \label{eq:abstract-iota-modulus}
\end{equation}
Choose $T_1\geq T$ with $r(T_1)<\varepsilon/2$. The relation $p(x)\rightsquigarrow_\Psi\eta$ provides a $(\delta,T_1)$-chain
\begin{equation}
 p(x)=y_0,y_1,\ldots,y_m=\eta,\qquad t_1,\ldots,t_m\geq T_1.
 \label{eq:abstract-base-chain}
\end{equation}
Define the lifted sequence by
\begin{equation}
 x_0=x,\qquad x_j=i(y_j),\quad j=1,\ldots,m.
 \label{eq:abstract-lifted-chain}
\end{equation}
The first leg begins at the prescribed ambient point $x$. By the triangle inequality, the tracking estimate, and \eqref{eq:abstract-iota-modulus},
\begin{equation}
 d_X(\Phi_{t_1}x_0,x_1)
 \leq d_X\bigl(\Phi_{t_1}x,i(\Psi_{t_1}p(x))\bigr)
      +d_X\bigl(i(\Psi_{t_1}p(x)),i(y_1)\bigr)
 <\varepsilon.
 \label{eq:abstract-first-lifted-jump}
\end{equation}
For every $j=2,\ldots,m$, equivariance of the section gives
\begin{equation}
 d_X(\Phi_{t_j}x_{j-1},x_j)
 =d_X\bigl(i(\Psi_{t_j}y_{j-1}),i(y_j)\bigr)
 <\frac{\varepsilon}{2}<\varepsilon.
 \label{eq:abstract-subsequent-lifted-jumps}
\end{equation}
All flight times are at least $T_1\geq T$, and the terminal point is $x_m=i(\eta)=y$. Hence the lifted sequence is an $(\varepsilon,T)$-chain from $x$ to $y$. Since $\varepsilon$ and $T$ were arbitrary, $x\rightsquigarrow_\Phi y$. Taking the set of all admissible targets $y$ gives \eqref{eq:abstract-chain-attainable-identity} directly from \eqref{eq:abstract-chain-reduction}.
\end{pf}

\begin{corollary}
\label{cor:abstract-recurrent-reduction}
Under Assumption~\ref{ass:abstract-tracking},
\begin{equation}
 \CR(\Phi)=i\bigl(\CR(\Psi)\bigr).
 \label{eq:abstract-CR-identity}
\end{equation}
Moreover, $i$ induces a bijection between the chain components of $\Psi$ and those of $\Phi$, and this bijection preserves and reflects the chain order between components.
\end{corollary}

\begin{pf}
Apply Theorem~\ref{thm:abstract-chain-reduction} with $y=x$. If $x\rightsquigarrow_\Phi x$, then the theorem forces $x=i(\eta)$ for some $\eta\in Y$ and gives $\eta\rightsquigarrow_\Psi\eta$. Thus $x\in i(\CR(\Psi))$. Conversely, if $\eta\in\CR(\Psi)$, then Theorem~\ref{thm:abstract-chain-reduction} applied to $x=i(\eta)$ and $y=i(\eta)$ gives $i(\eta)\rightsquigarrow_\Phi i(\eta)$. Hence \eqref{eq:abstract-CR-identity} holds.

For arbitrary $\xi,\eta\in Y$, the same theorem yields the exact equivalence
\begin{equation}
 i(\xi)\rightsquigarrow_\Phi i(\eta)
 \quad\Longleftrightarrow\quad
 \xi\rightsquigarrow_\Psi\eta.
 \label{eq:abstract-section-relation-equivalence}
\end{equation}
Therefore mutual ambient chain attainability is preserved and reflected by $i$. The map $i$ consequently induces a bijection of chain components. Since \eqref{eq:abstract-section-relation-equivalence} also preserves one-way attainability, it preserves and reflects the order induced on those components.
\end{pf}

\begin{remark}
\label{rem:abstract-source-target-asymmetry}
Uniform tracking does not force a chain source to lie on the equivariant section. It forces only the target. This distinction is essential for non-surjective semiflows and reappears below in the asymmetry between graph attractors and full-fiber dual repellers.
\end{remark}

\section{Invariant graph and exponential tracking}
\label{sec:graph}

We first derive the geometric consequences of finite-window contraction. The proofs are included because the same estimates will be needed later in the chain-recurrence argument.

\subsection{Graph transform}
\label{subsec:graph-transform}

Let \(C(\Omega,K)\) be the space of continuous maps \(g:\Omega\to K\), equipped with
\begin{equation}
 d_\infty(g_1,g_2)
 :=\sup_{\omega\in\Omega}
 \lvert g_1(\omega)-g_2(\omega)\rvert_\omega.
 \label{eq:section-sup-metric}
\end{equation}
Because the moving norms are uniformly equivalent to the Euclidean norm, \(d_\infty\) is equivalent to the usual uniform metric; hence \(C(\Omega,K)\) is complete. Define the graph transform \(\Gcal:C(\Omega,K)\to C(\Omega,K)\) by
\begin{equation}
 (\Gcal g)(\omega)
 :=\varphi\bigl(\tau,\theta_{-\tau}\omega, 
 g(\theta_{-\tau}\omega)\bigr).
 \label{eq:graph-transform}
\end{equation}

\begin{lemma}
\label{lem:graph-transform-contraction}
Under Assumption~\ref{ass:well-posedness} and Definition~\ref{def:block-contraction}, the graph transform \(\Gcal\) is a contraction with factor \(q\):
\begin{equation}
 d_\infty(\Gcal g_1,\Gcal g_2)
 \leq qd_\infty(g_1,g_2).
 \label{eq:graph-transform-contraction}
\end{equation}
\end{lemma}

\begin{pf}
For every \(\omega\in\Omega\), apply \eqref{eq:block-contraction} with base point \(\theta_{-\tau}\omega\). Since \(\theta_\tau(\theta_{-\tau}\omega)=\omega\),
\begin{equation}
 \left\lvert(\Gcal g_1)(\omega)-(\Gcal g_2)(\omega)\right\rvert_\omega
 \leq q\left\lvert g_1(\theta_{-\tau}\omega)-g_2(\theta_{-\tau}\omega)\right\rvert_{\theta_{-\tau}\omega}.
 \label{eq:graph-transform-proof-estimate}
\end{equation}
Taking the supremum over \(\omega\) proves \eqref{eq:graph-transform-contraction}.
\end{pf}

\begin{theorem}
\label{thm:invariant-graph}
Suppose that the cocycle is uniformly contracting over a finite window. Then there exists a unique continuous map \(h:\Omega\to K\) satisfying
\begin{equation}
 \varphi(t,\omega,h(\omega))
 =h(\theta_t\omega),
 \qquad t\geq0,
 \label{eq:full-graph-invariance}
\end{equation}
for every \(\omega\in\Omega\). Equivalently, the graph
\begin{equation}
 \Gamma_h:=\graph(h)
 =\{(\omega,h(\omega)):\omega\in\Omega\}
 \label{eq:Gamma-h}
\end{equation}
is compact and invariant under \(\Pi\). Moreover, \(\Pi|_{\Gamma_h}\) is topologically conjugate to \(\theta\) through the homeomorphism
\begin{equation}
 \iota:\Omega\to\Gamma_h,
 \qquad \iota(\omega)=(\omega,h(\omega)).
 \label{eq:graph-embedding}
\end{equation}
\end{theorem}

\begin{pf}
Lemma~\ref{lem:graph-transform-contraction} and the Banach fixed-point theorem give a unique \(h\in C(\Omega,K)\) satisfying \(\Gcal h=h\), or equivalently
\begin{equation}
 h(\theta_\tau\omega)=\varphi(\tau,\omega,h(\omega)).
 \label{eq:tau-invariance}
\end{equation}
It remains to prove invariance for arbitrary times. For fixed \(s\geq0\), define
\begin{equation}
 h_s(\omega)
 :=\varphi\bigl(s,\theta_{-s}\omega,h(\theta_{-s}\omega)\bigr).
 \label{eq:shifted-section}
\end{equation}
Using the cocycle identity and \eqref{eq:tau-invariance}, one verifies that \(\Gcal h_s=h_s\). By uniqueness of the fixed point of \(\Gcal\), \(h_s=h\). Replacing \(\omega\) by \(\theta_s\omega\) in \eqref{eq:shifted-section} yields \eqref{eq:full-graph-invariance}. The conjugacy relation
\begin{equation}
 \Pi_t\circ\iota=\iota\circ\theta_t
 \label{eq:graph-conjugacy}
\end{equation}
then follows immediately.
\end{pf}

The invariant graph can also be represented as a pullback limit. For any continuous section \(g_0\in C(\Omega,K)\),
\begin{equation}
 h(\omega)
 =\lim_{j\to\infty}
 \varphi\bigl(j\tau,\theta_{-j\tau}\omega,
 g_0(\theta_{-j\tau}\omega)\bigr),
 \label{eq:pullback-graph-representation}
\end{equation}
uniformly in \(\omega\).

\subsection{Uniform exponential tracking}
\label{subsec:tracking}

The finite-window contraction established above extends to a uniform
continuous-time exponential estimate. We first recall the notion of a global
attractor and then show that the invariant graph is the unique global attractor
of the skew-product semiflow.

\begin{definition}
\label{def:global-attractor}
For a continuous semiflow $\Phi$ on a compact metric space $X$, a global attractor is a compact strictly invariant set $\mathcal A$ that attracts $X$ uniformly and is minimal among compact sets with this property. Equivalently,
\begin{equation}
 \mathcal A
 =\bigcap_{s\geq0}
 \overline{\bigcup_{t\geq s}\Phi_t(X)}.
 \label{eq:global-attractor-definition}
\end{equation}
\end{definition}

\begin{theorem}
\label{thm:exponential-tracking}
Under the assumptions of Theorem~\ref{thm:invariant-graph}, let
\begin{equation}
 \lambda:=\frac{-\log q}{\tau}>0,
 \qquad
 C:=\frac{L_\tau}{q}.
 \label{eq:tracking-constants}
\end{equation}
Then
\begin{equation}
 \left\lvert
 \varphi(t,\omega,x)-h(\theta_t\omega)
 \right\rvert_{\theta_t\omega}
 \leq
 C\e^{-\lambda t}
 \lvert x-h(\omega)\rvert_\omega
 \label{eq:exponential-tracking}
\end{equation}
for every \(t\geq0\), \(\omega\in\Omega\), and \(x\in K\). Moreover, if \(d_H\) denotes the Hausdorff distance induced by the fixed product metric \(D\), then, with \(c_P:=\max\{1,m_P^{-1/2}\}\),
\begin{equation}
 d_H\bigl(\Pi_t(Z),\Gamma_h\bigr)
 \leq c_PC\e^{-\lambda t}\diam_P(K),
 \label{eq:global-attractor-Hausdorff}
\end{equation}
where
\begin{equation}
 \diam_P(K):=
 \sup_{\omega\in\Omega}\sup_{x,y\in K}\lvert x-y\rvert_\omega.
 \label{eq:moving-diameter}
\end{equation}
Consequently, \(\Gamma_h\) is the unique global attractor of \(\Pi\) on \(Z\).
\end{theorem}

\begin{pf}
Write \(t=j\tau+s\), where \(j\in\N\cup\{0\}\) and \(s\in[0,\tau)\). By the cocycle property, graph invariance, and \(j\) successive applications of \eqref{eq:block-contraction},
\begin{equation}
 \left\lvert\varphi(j\tau,\omega,x)-h(\theta_{j\tau}\omega)\right\rvert_{\theta_{j\tau}\omega}
 \leq q^j\lvert x-h(\omega)\rvert_\omega.
 \label{eq:integer-window-tracking}
\end{equation}
Applying the intermediate-time estimate \eqref{eq:intermediate-lipschitz} over the remaining interval of length \(s\) gives
\begin{equation}
 \left\lvert\varphi(t,\omega,x)-h(\theta_t\omega)\right\rvert_{\theta_t\omega}
 \leq L_\tau q^j\lvert x-h(\omega)\rvert_\omega.
 \label{eq:tracking-block-estimate}
\end{equation}
Since \(j\tau\geq t-\tau\),
\begin{equation}
 q^j=\e^{-\lambda j\tau}
 \leq \e^{-\lambda(t-\tau)}
 =q^{-1}\e^{-\lambda t},
 \label{eq:block-to-exponential}
\end{equation}
which proves \eqref{eq:exponential-tracking}.

For every \(z=(\omega,x)\in Z\), the point \(\Pi_tz\) has the same base coordinate as \(\iota(\theta_t\omega)\), and the fiber distance between them is bounded by \eqref{eq:exponential-tracking}. This gives one side of \eqref{eq:global-attractor-Hausdorff}. Conversely, because the base is a flow, every graph point satisfies
\begin{equation}
 \iota(\eta)
 =\Pi_t\bigl(\iota(\theta_{-t}\eta)\bigr)
 \in\Pi_t(Z),
 \label{eq:graph-contained-image}
\end{equation}
so \(\Gamma_h\subset\Pi_t(Z)\) and the opposite Hausdorff semidistance is zero.

Finally, if \(A\subset Z\) is compact and attracts \(Z\), then for any \(\eta\in\Omega\), any \(x_0\in K\), and \(z_t=(\theta_{-t}\eta,x_0)\),
\begin{equation}
 \dist\bigl(\iota(\eta),A\bigr)
 \leq
 D\bigl(\iota(\eta),\Pi_tz_t\bigr)
 +\dist\bigl(\Pi_tz_t,A\bigr).
 \label{eq:global-attractor-minimality}
\end{equation}
The first term tends to zero uniformly by tracking and the second by attraction of \(Z\). Hence \(\Gamma_h\subset A\). Therefore \(\Gamma_h\) is the minimal compact set attracting \(Z\), and thus the unique global attractor in the sense of Definition~\ref{def:global-attractor}.
\end{pf}

Define the graph-distance function
\begin{equation}
 V(\omega,x)
 :=\lvert x-h(\omega)\rvert_\omega.
 \label{eq:graph-distance-V}
\end{equation}
It is continuous on \(Z\) and satisfies
\begin{equation}
 V(\Pi_t(\omega,x))
 \leq C\e^{-\lambda t}V(\omega,x).
 \label{eq:V-decay}
\end{equation}
Although \(V\) need not decrease at every instant, it decreases by a common factor over every sufficiently long flight time. The next section converts this discrete Lyapunov property into quantitative estimates for arbitrary chains.

\subsection{Limit sets}
\label{subsec:limit-sets}

For \(\omega\in\Omega\), denote by \(\omega_\theta(\omega)\) its omega-limit set under the base flow. For \(z\in Z\), write \(\omega_\Pi(z)\) for the omega-limit set under the skew product.

\begin{proposition}
\label{prop:omega-limit-lifting}
For every \((\omega,x)\in Z\),
\begin{equation}
 \omega_\Pi(\omega,x)
 =\iota\bigl(\omega_\theta(\omega)\bigr)
 =\{(\eta,h(\eta)):\eta\in\omega_\theta(\omega)\}.
 \label{eq:omega-limit-identity}
\end{equation}
\end{proposition}

\begin{pf}
Let \(\eta\in\omega_\theta(\omega)\). There is a sequence \(t_j\to\infty\) such that \(\theta_{t_j}\omega\to\eta\). Exponential tracking and continuity of \(h\) imply
\begin{equation}
 \Pi_{t_j}(\omega,x)
 =\bigl(\theta_{t_j}\omega,\varphi(t_j,\omega,x)\bigr)
 \longrightarrow(\eta,h(\eta)),
 \label{eq:omega-lift-convergence}
\end{equation}
so the right-hand side of \eqref{eq:omega-limit-identity} is contained in the left-hand side.

Conversely, let \(z_*\in\omega_\Pi(\omega,x)\). For some \(t_j\to\infty\), \(\Pi_{t_j}(\omega,x)\to z_*\). Projection gives \(\theta_{t_j}\omega\to\eta:=\pi_\Omega(z_*)\), hence \(\eta\in\omega_\theta(\omega)\). Tracking gives
\begin{equation}
 D\bigl(\Pi_{t_j}(\omega,x),\iota(\theta_{t_j}\omega)\bigr)\to0,
 \label{eq:omega-limit-tracking-distance}
\end{equation}
and therefore \(z_*=\iota(\eta)\). This proves the reverse inclusion.
\end{pf}

\section{Chain attainability and attractor--repeller structure}
\label{sec:recurrent}

The abstract source--target theorem already determines the ambient chain relation. We first derive a quantitative estimate that measures how close every approximate chain target must lie to the invariant graph and then record the controlled skew-product specialization of the abstract theorem. The remaining subsections identify the recurrent core, construct complete Lyapunov functions, and describe the attractor--repeller asymmetry created by a non-surjective contracting semiflow.

\subsection{A quantitative estimate along chains}
\label{subsec:chain-estimate}

For $\varepsilon\geq0$, define the modulus of continuity of $V$ by
\begin{equation}
 \omega_V(\varepsilon)
 :=\sup\{\lvert V(z)-V(w)\rvert:
 z,w\in Z,\ D(z,w)\leq\varepsilon\}.
 \label{eq:V-modulus-continuity}
\end{equation}
Compactness of $Z$ implies $\omega_V(\varepsilon)\to0$ as $\varepsilon\downarrow0$.

\begin{lemma}
\label{lem:chain-distance-estimate}
Choose $T>0$ such that
\begin{equation}
 \rho_T:=C\e^{-\lambda T}<1.
 \label{eq:rho-T-definition}
\end{equation}
Let $z_0,\ldots,z_m$ be an $(\varepsilon,T)$-chain for $\Pi$ with $m\geq1$. Then
\begin{equation}
 V(z_j)
 \leq \rho_T^jV(z_0)
 +\omega_V(\varepsilon)
 \frac{1-\rho_T^j}{1-\rho_T},
 \qquad j=1,\ldots,m.
 \label{eq:chain-distance-iteration}
\end{equation}
In particular,
\begin{equation}
 V(z_m)
 \leq \rho_TV(z_0)
 +\frac{\omega_V(\varepsilon)}{1-\rho_T},
 \label{eq:chain-endpoint-bound}
\end{equation}
and if the chain returns to its initial point, then
\begin{equation}
 V(z_0)
 \leq\frac{\omega_V(\varepsilon)}{1-\rho_T}.
 \label{eq:return-chain-distance-bound}
\end{equation}
\end{lemma}

\begin{pf}
For every chain leg, $t_i\geq T$, so \eqref{eq:V-decay} and \eqref{eq:rho-T-definition} imply
\begin{equation}
 V(\Pi_{t_i}z_{i-1})
 \leq\rho_TV(z_{i-1}).
 \label{eq:chain-leg-contraction}
\end{equation}
The jump condition and the definition of $\omega_V$ give
\begin{equation}
 V(z_i)
 \leq V(\Pi_{t_i}z_{i-1})+\omega_V(\varepsilon)
 \leq\rho_TV(z_{i-1})+\omega_V(\varepsilon).
 \label{eq:chain-recursion-detailed}
\end{equation}
Iteration proves \eqref{eq:chain-distance-iteration}. Since $m\geq1$, one has $\rho_T^m\leq\rho_T$, and \eqref{eq:chain-endpoint-bound} follows. If $z_m=z_0$, then
\begin{equation}
 (1-\rho_T^m)V(z_0) 
 \leq
 \omega_V(\varepsilon)\frac{1-\rho_T^m}{1-\rho_T},
 \label{eq:return-chain-cancellation}
\end{equation}
which yields \eqref{eq:return-chain-distance-bound}.
\end{pf}

\begin{remark}
\label{rem:quantitative-chain-exclusion}
If $V(z)=r>0$, choose $\varepsilon>0$ so that
\begin{equation}
 \omega_V(\varepsilon)<\frac{1}{2}(1-\rho_T)r.
 \label{eq:explicit-no-return-tolerance}
\end{equation}
Then no $(\varepsilon,T)$-chain can return from $z$ to itself. More generally, \eqref{eq:chain-endpoint-bound} bounds the graph distance of every possible chain target in terms of the graph distance of the source.
\end{remark}

\subsection{Application to the controlled skew product}
\label{subsec:ambient-chain-relation}

Recall that, for a continuous semiflow \(\Phi\),
\[
\mathcal C_\Phi^+(x)
=
\{y:x\rightsquigarrow_\Phi y\}
\]
denotes the positive chain-attainable set from \(x\).
The abstract reduction theorem now specializes directly to the controlled skew product and describes these sets for every source point of $Z$, not only for sources on the invariant graph.

\begin{corollary}
\label{cor:controlled-chain-reduction}
Suppose that the cocycle is uniformly contracting over a finite window. Let $z=(\omega,x)\in Z$ and $w\in Z$. Then
\begin{equation}
  z\rightsquigarrow_\Pi w
 \quad\Longleftrightarrow\quad
 w=\iota(\eta)\ \text{for some }\eta\in\Omega
 \text{ with }\omega\rightsquigarrow_\theta\eta.
 \label{eq:ambient-chain-relation}
\end{equation}
Equivalently,
\begin{equation}
  \mathcal C_\Pi^+(z)
 =\iota\bigl(\mathcal C_\theta^+(\pi_\Omega z)\bigr).
 \label{eq:chain-attainable-set-identity}
\end{equation}
\end{corollary}

\begin{pf}
Apply Theorem~\ref{thm:abstract-chain-reduction} with
\begin{equation}
 X=Z,
 \quad Y=\Omega,
 \quad \Phi=\Pi,
 \quad \Psi=\theta,
 \quad p=\pi_\Omega,
 \quad i=\iota.
 \label{eq:controlled-abstract-identification}
\end{equation}
The factor and section identities are \eqref{eq:projection-semiconjugacy} and \eqref{eq:graph-conjugacy}. Exponential tracking \eqref{eq:exponential-tracking}, compactness of $Z$, and equivalence of the fiber and product metrics imply the uniform tracking condition \eqref{eq:abstract-uniform-tracking}.
\end{pf}

\begin{remark}
\label{rem:source-target-asymmetry}
The source of a chain may lie anywhere in the fiber, but every chain target lies on $\Gamma_h$. This asymmetry is already visible for the scalar system $\dot x=-x$: every $x$ chain-attains $0$, whereas no nonzero point can be a chain target. Corollary~\ref{cor:controlled-chain-reduction} shows that the same phenomenon holds over an arbitrary compact base and that the admissible graph targets are selected exactly by the base chain relation.
\end{remark}

\subsection{The recurrent core, chain components, and Conley order}
\label{subsec:recurrent-core-order}

The ambient chain-reduction formula now determines the entire recurrent
structure of the skew product. We first identify the chain-recurrent set and
then show that the chain components, together with their induced Conley order,
are transported exactly from the base dynamics.

\begin{theorem}
\label{thm:chain-recurrence-identity}
Under finite-window contraction,
\begin{equation}
  \CR(\Pi)
 =\iota\bigl(\CR(\theta)\bigr)
 =\{(\omega,h(\omega)):\omega\in\CR(\theta)\}.
 \label{eq:CR-identity}
\end{equation}
\end{theorem}

\begin{pf}
Apply Corollary~\ref{cor:controlled-chain-reduction} with $w=z$. The equality $z=\iota(\eta)$ forces $z$ to lie on the graph and $\eta=\pi_\Omega(z)$. Hence
\begin{equation}
 z\rightsquigarrow_\Pi z
 \quad\Longleftrightarrow\quad
 z=\iota(\omega)\ \text{and }\omega\rightsquigarrow_\theta\omega,
 \label{eq:CR-from-ambient-relation}
\end{equation}
which is precisely \eqref{eq:CR-identity}.
\end{pf}

\begin{theorem}
\label{thm:chain-component-correspondence}
For arbitrary $\omega_1,\omega_2\in\Omega$,
\begin{equation}
 \omega_1\rightsquigarrow_\theta\omega_2
 \quad\Longleftrightarrow\quad
 \iota(\omega_1)\rightsquigarrow_\Pi\iota(\omega_2).
 \label{eq:ambient-chain-relation-equivalence}
\end{equation}
Consequently, $C\mapsto\iota(C)$ is a bijection between the chain components of the base and those of the skew product. Moreover, it preserves and reflects the chain order between components: if
\begin{equation}
 C_1\preceq_\theta C_2
 \quad\Longleftrightarrow\quad
 \omega_1\rightsquigarrow_\theta\omega_2
 \text{ for some }\omega_i\in C_i,
 \label{eq:base-component-order}
\end{equation}
then
\begin{equation}
 C_1\preceq_\theta C_2
 \quad\Longleftrightarrow\quad
 \iota(C_1)\preceq_\Pi\iota(C_2).
 \label{eq:component-order-isomorphism}
\end{equation}
\end{theorem}

\begin{pf}
Identity \eqref{eq:ambient-chain-relation-equivalence} is the restriction of Corollary~\ref{cor:controlled-chain-reduction} to graph sources and graph targets. Theorem~\ref{thm:chain-recurrence-identity} identifies the recurrent sets, so mutual chain attainability is preserved and reflected by $\iota$. This proves the component bijection. Since the chain relation is transitive, the relation in \eqref{eq:base-component-order} is independent of the chosen representatives, and \eqref{eq:component-order-isomorphism} follows again from \eqref{eq:ambient-chain-relation-equivalence}.
\end{pf}

\begin{remark}
\label{rem:control-interpretation-CR}
Every equilibrium, periodic orbit, minimal set, or more general chain component of the closed loop is the graph lift of a corresponding recurrent object of the exosystem. The controlled fiber contributes no additional chain-recurrent class, cannot merge two distinct base components, and cannot alter the directed chain order between them.
\end{remark}

\begin{corollary}
\label{cor:chain-transitivity}
The set $\CR(\Pi)$ consists of a single chain component if and only if $\CR(\theta)$ consists of a single chain component. If the base is minimal, then $\Gamma_h=\CR(\Pi)$ is minimal and is the unique global attractor.
\end{corollary}

\subsection{Constructive lifting of complete Lyapunov functions}
\label{subsec:complete-lyapunov}

We use a metric-semiflow version of Conley's complete Lyapunov function; compare \cite{Conley1978,Hurley1995,Patrao2011}. A continuous function $L:X\to\R$ for a continuous semiflow $\Phi$ on a compact metric space is called complete if:
\begin{enumerate}
 \item $L(\Phi_t x)\leq L(x)$ for every $x\in X$ and $t\geq0$;
 \item $L(\Phi_t x)<L(x)$ for every $x\notin\CR(\Phi)$ and every $t>0$;
 \item $L$ is constant on each chain component, takes distinct values on distinct chain components, and $L(\CR(\Phi))$ is nowhere dense in $\R$.
\end{enumerate}
Lyapunov functions for lifted finite Morse decompositions of skew-product semiflows are constructed in \cite[Theorem~5.7]{BortolanEtAl2013}. The result below works directly at the chain-recurrent level and adds an explicit weighted-supremum transverse gauge whose decay is uniform along the full semiflow.

\begin{theorem}
\label{thm:complete-lyapunov-lift}
Assume finite-window contraction and let $L_\theta:\Omega\to\R$ be a complete Lyapunov function for the base flow. Fix $\alpha\in(0,\lambda)$ and define
\begin{equation}
 W_\alpha(z)
 :=\sup_{t\geq0}\e^{\alpha t}V(\Pi_tz),
 \qquad z\in Z,
 \label{eq:adapted-transverse-gauge}
\end{equation}
where $V$ is the graph-distance function and $\lambda$ is the exponential tracking rate in Theorem~\ref{thm:exponential-tracking}. Then $W_\alpha$ is continuous,
\begin{equation}
 W_\alpha(z)=0
 \quad\Longleftrightarrow\quad
 z\in\Gamma_h,
 \label{eq:transverse-gauge-zero-set}
\end{equation}
and
\begin{equation}
 W_\alpha(\Pi_sz)
 \leq\e^{-\alpha s}W_\alpha(z),
 \qquad s\geq0.
 \label{eq:transverse-gauge-decay}
\end{equation}
For every $\gamma>0$, the function
\begin{equation}
 L_\Pi(z)
 :=L_\theta(\pi_\Omega z)+\gamma W_\alpha(z)
 \label{eq:complete-lyapunov-lift}
\end{equation}
is a complete Lyapunov function for $\Pi$. Moreover,
\begin{equation}
 L_\Pi\circ\iota=L_\theta,
 \qquad
 L_\Pi(\CR(\Pi))=L_\theta(\CR(\theta)).
 \label{eq:complete-lyapunov-values}
\end{equation}
Conversely, the restriction of any complete Lyapunov function of $\Pi$ to $\Gamma_h$, pulled back by $\iota$, is a complete Lyapunov function of the base.
\end{theorem}

\begin{pf}
We verify the properties of $W_\alpha$ and then the three defining properties of completeness.

First, the exponential tracking estimate gives
\begin{equation}
 \e^{\alpha t}V(\Pi_tz)
 \leq C\e^{-(\lambda-\alpha)t}V(z),
 \qquad t\geq0.
 \label{eq:weighted-tracking-bound}
\end{equation}
Since $\lambda-\alpha>0$, the right-hand side is bounded uniformly in $t$, and therefore $W_\alpha(z)<\infty$ for every $z\in Z$. The value $t=0$ is included in the supremum, so
\begin{equation}
 W_\alpha(z)\geq V(z).
 \label{eq:gauge-dominates-distance}
\end{equation}
If $W_\alpha(z)=0$, then $V(z)=0$ and hence $z\in\Gamma_h$. Conversely, graph invariance gives $V(\Pi_tz)=0$ for every $t\geq0$ whenever $z\in\Gamma_h$. This proves \eqref{eq:transverse-gauge-zero-set}.

For continuity, define the finite-horizon maxima
\begin{equation}
 W_{\alpha,T}(z)
 :=\max_{0\leq t\leq T}\e^{\alpha t}V(\Pi_tz).
 \label{eq:truncated-transverse-gauge}
\end{equation}
The function $(t,z)\mapsto \e^{\alpha t}V(\Pi_tz)$ is continuous on the compact set $[0,T]\times Z$. The maximum theorem therefore implies that $W_{\alpha,T}$ is continuous. Let $M_V:=\max_{z\in Z}V(z)$. From \eqref{eq:weighted-tracking-bound},
\begin{equation}
 0\leq W_\alpha(z)-W_{\alpha,T}(z)
 \leq \sup_{t\geq T}\e^{\alpha t}V(\Pi_tz)
 \leq CM_V\e^{-(\lambda-\alpha)T}.
 \label{eq:uniform-tail-transverse-gauge}
\end{equation}
The final bound is independent of $z$ and converges to zero as $T\to\infty$. Thus $W_{\alpha,T}\to W_\alpha$ uniformly on $Z$, and $W_\alpha$ is continuous.

Next, for $s\geq0$, the semigroup property yields
\begin{equation}
 W_\alpha(\Pi_sz)
 =\sup_{t\geq0}\e^{\alpha t}V(\Pi_{t+s}z)
 =\e^{-\alpha s}\sup_{u\geq s}\e^{\alpha u}V(\Pi_uz)
 \leq\e^{-\alpha s}W_\alpha(z).
 \label{eq:transverse-gauge-shift-proof}
\end{equation}
This proves \eqref{eq:transverse-gauge-decay}. In particular, if $s>0$ and $z\notin\Gamma_h$, then $W_\alpha(z)>0$ and
\begin{equation}
 W_\alpha(\Pi_sz)<W_\alpha(z).
 \label{eq:strict-transverse-gauge-decay}
\end{equation}

We now verify completeness of $L_\Pi$. For every $t\geq0$, the factor relation gives
\begin{equation}
 L_\theta\bigl(\pi_\Omega(\Pi_tz)\bigr)
 =L_\theta(\theta_t\pi_\Omega z)
 \leq L_\theta(\pi_\Omega z),
 \label{eq:base-lyapunov-monotonicity-lift}
\end{equation}
while \eqref{eq:transverse-gauge-decay} shows that the transverse term is also nonincreasing. Hence $L_\Pi(\Pi_tz)\leq L_\Pi(z)$.

To prove strict decrease outside $\CR(\Pi)$, consider two cases. If $z\notin\Gamma_h$, then \eqref{eq:strict-transverse-gauge-decay} is strict for every $t>0$, whereas the base term cannot increase. Consequently,
\begin{equation}
 L_\Pi(\Pi_tz)<L_\Pi(z),\qquad t>0.
 \label{eq:complete-lift-strict-off-graph}
\end{equation}
If $z\in\Gamma_h\setminus\CR(\Pi)$, write $z=\iota(\omega)$. Theorem~\ref{thm:chain-recurrence-identity} gives $\omega\notin\CR(\theta)$. The transverse term vanishes along the whole graph orbit, and completeness of $L_\theta$ gives
\begin{equation}
 L_\Pi(\Pi_tz)
 =L_\theta(\theta_t\omega)
 <L_\theta(\omega)
 =L_\Pi(z),\qquad t>0.
 \label{eq:complete-lift-strict-on-graph}
\end{equation}
These two cases cover $Z\setminus\CR(\Pi)$.

Finally, on the recurrent set one has
\begin{equation}
 \CR(\Pi)=\iota(\CR(\theta)),
 \qquad
 W_\alpha|_{\CR(\Pi)}=0,
 \qquad
 L_\Pi\circ\iota=L_\theta.
 \label{eq:complete-lift-recurrent-restriction}
\end{equation}
The chain components correspond bijectively under $\iota$. Hence $L_\Pi$ is constant on every chain component, assigns different values to different components, and satisfies
\begin{equation}
 L_\Pi(\CR(\Pi))=L_\theta(\CR(\theta)).
 \label{eq:complete-lift-recurrent-values-proof}
\end{equation}
The set on the right is nowhere dense because $L_\theta$ is complete. This proves all parts of the stated definition and establishes \eqref{eq:complete-lyapunov-values}.

For the converse, let $\widetilde L_\Pi$ be complete for $\Pi$ and define $L_\theta:=\widetilde L_\Pi\circ\iota$. Equivariance gives
\begin{equation}
 L_\theta(\theta_t\omega)
 =\widetilde L_\Pi(\Pi_t\iota(\omega))
 \leq\widetilde L_\Pi(\iota(\omega))
 =L_\theta(\omega).
 \label{eq:complete-lift-converse-monotonicity}
\end{equation}
If $\omega\notin\CR(\theta)$, then $\iota(\omega)\notin\CR(\Pi)$ and the inequality is strict for every $t>0$. The component correspondence and the identity $\CR(\Pi)=\iota(\CR(\theta))$ transfer the remaining component and recurrent-value conditions. Thus $L_\theta$ is complete.
\end{pf}

\begin{remark}
\label{rem:supremum-gauge}
The original graph distance need not decrease monotonically during every short time interval because the system may be transiently expansive. The adapted gauge \eqref{eq:adapted-transverse-gauge} converts eventual exponential tracking into the pointwise semigroup inequality \eqref{eq:transverse-gauge-decay}. It therefore separates the transient fiber direction without imposing instantaneous contraction.
\end{remark}

\subsection{Correspondence of compact strictly invariant sets}
\label{subsec:compact-invariant-sets}

For a semiflow, forward invariance and strict invariance are different. A compact set $C\subset Z$ is called strictly invariant if
\begin{equation}
 \Pi_t(C)=C,
 \qquad t\geq0.
 \label{eq:strict-invariance-definition}
\end{equation}
For the base flow, a compact set $A\subset\Omega$ is strictly invariant if $\theta_t(A)=A$ for all $t\in\R$, equivalently for all $t\geq0$.

\begin{theorem}
\label{thm:compact-invariant-correspondence}
Suppose that finite-window contraction holds. Then every compact strictly invariant set $C\subset Z$ is contained in $\Gamma_h$. Moreover,
\begin{equation}
 A\longmapsto\widehat A:=\iota(A),
 \qquad
 C\longmapsto\pi_\Omega(C),
 \label{eq:invariant-set-bijection}
\end{equation}
are mutually inverse inclusion-preserving bijections between compact strictly invariant sets of the base flow and compact strictly invariant sets of the skew-product semiflow.
\end{theorem}

\begin{pf}
Let $S:=\Pi_\tau$. Finite-window contraction and graph invariance give
\begin{equation}
 V(Sz)\leq qV(z),
 \qquad z\in Z.
 \label{eq:one-step-V-contraction}
\end{equation}
Fix $z\in C$. Since $S^n(C)=C$, for every $n\geq1$ there exists $z_{-n}\in C$ such that $S^nz_{-n}=z$. Hence
\begin{equation}
 V(z)
 =V(S^nz_{-n})
 \leq q^nV(z_{-n})
 \leq q^n\max_{w\in C}V(w).
 \label{eq:backward-invariant-V-estimate}
\end{equation}
The maximum is finite by compactness. Letting $n\to\infty$ gives $V(z)=0$, and therefore $C\subset\Gamma_h$.

If $A\subset\Omega$ is compact and strictly invariant, graph conjugacy gives $\Pi_t(\iota(A))=\iota(A)$. Conversely, if $C\subset Z$ is compact and strictly invariant, then $C\subset\Gamma_h$, the set $A=\pi_\Omega(C)$ is compact and strictly invariant, and $C=\iota(A)$. Thus the two maps in \eqref{eq:invariant-set-bijection} are mutually inverse and preserve inclusions.
\end{pf}

\begin{remark}
\label{rem:strict-vs-forward-invariance}
The conclusion is false for merely forward-invariant sets. A sufficiently large closed tube around $\Gamma_h$ may be mapped strictly into itself while containing off-graph points. The backward selections supplied by $S^n(C)=C$ are the decisive ingredient in \eqref{eq:backward-invariant-V-estimate}.
\end{remark}

\subsection{Attractors, standard dual repellers, and strict invariant cores}
\label{subsec:attractors}

The following metric definition is convenient for the basin estimates below. On compact metric spaces, continuous semiflows are proper, and this definition agrees with the standard attracting-neighborhood characterization $A=\omega(U)=\operatorname{Inv}_\Phi(U)$ after replacing $U$ by a trapping neighborhood; see \cite{KaliesVandervorst2025}.

\begin{definition}
\label{def:compact-attractor}
For a continuous semiflow $\Phi$ on a compact metric space $X$, a compact strictly invariant set $A\subset X$ is an attractor if it has an open neighborhood $U$ such that
\begin{equation}
 \lim_{t\to\infty}
 \sup_{x\in U}\dist\bigl(\Phi_t x,A\bigr)=0.
 \label{eq:uniform-neighborhood-attraction}
\end{equation}
Its basin is
\begin{equation}
 \mathcal B_\Phi(A)
 :=\{x\in X:\dist(\Phi_t x,A)\to0\}.
 \label{eq:basin-definition}
\end{equation}
The standard dual repeller of $A$ is
\begin{equation}
 A_\Phi^*
 :=\{x\in X:\omega_\Phi(x)\cap A=\varnothing\},
 \label{eq:dual-repeller-definition}
\end{equation}
where $\omega_\Phi(x)$ is the omega-limit set of $x$. For a set $Q\subset X$, we write $\Core_\Phi^{\mathrm{str}}(Q)$ for its largest compact strictly invariant subset when such a largest subset exists. In the application below, its existence and explicit form are part of the theorem.
\end{definition}

For attractors of continuous systems on compact spaces, the standard attractor--repeller theory identifies $A_\Phi^*$ with the corresponding repelling set determined by the complement of the basin; see \cite{KaliesMischaikowVandervorst2022,KaliesVandervorst2025}. In particular, for the base flow $\theta$, the set $A_\theta^*$ is compact and strictly invariant.

\begin{theorem}
\label{thm:attractor-correspondence}
The map $A\mapsto\widehat A=\iota(A)$ is an inclusion-preserving bijection between attractors of $\theta$ and attractors of $\Pi$. For every base attractor $A$,
\begin{equation}
  \mathcal B_\Pi(\widehat A)
 =\pi_\Omega^{-1}\bigl(\mathcal B_\theta(A)\bigr).
 \label{eq:basin-identity}
\end{equation}
Thus the attractor posets of the two systems are order isomorphic.
\end{theorem}

\begin{pf}
Let $z=(\omega,x)$. Exponential tracking gives
\begin{equation}
 D\bigl(\Pi_tz,\iota(\theta_t\omega)\bigr)
 \longrightarrow0.
 \label{eq:orbit-to-graph-orbit}
\end{equation}
If $\dist(\theta_t\omega,A)\to0$, uniform continuity of $\iota$ implies
\begin{equation}
 \dist\bigl(\iota(\theta_t\omega),\iota(A)\bigr)\to0,
 \label{eq:base-to-graph-attraction}
\end{equation}
and the triangle inequality yields $\dist(\Pi_tz,\widehat A)\to0$. Conversely, if $\dist(\Pi_tz,\widehat A)\to0$, uniform continuity of $\pi_\Omega$, together with $\pi_\Omega(\widehat A)=A$, gives $\dist(\theta_t\omega,A)\to0$. Hence \eqref{eq:basin-identity} holds.

Let $A$ be a base attractor with attracting neighborhood $U$. We verify uniform attraction of the full cylinder $\pi_\Omega^{-1}(U)$. Fix $\varepsilon>0$. By uniform continuity of $\iota$, choose $\delta>0$ so that
\[
 d_\Omega(\omega_1,\omega_2)<\delta
 \quad\Longrightarrow\quad
 D(\iota(\omega_1),\iota(\omega_2))<\frac{\varepsilon}{2}.
\]
Uniform attraction of $U$ to $A$ gives $T_A>0$ such that
\[
 \dist(\theta_t\omega,A)<\delta,
 \qquad \omega\in U,\quad t\geq T_A.
\]
Uniform graph tracking gives $T_G>0$ such that
\[
 D(\Pi_t(\omega,x),\iota(\theta_t\omega))<\frac{\varepsilon}{2},
 \qquad (\omega,x)\in Z,\quad t\geq T_G.
\]
For $t\geq\max\{T_A,T_G\}$ and $z=(\omega,x)\in\pi_\Omega^{-1}(U)$, choose $a_t\in A$ with $d_\Omega(\theta_t\omega,a_t)<\delta$. Then
\[
 \dist(\Pi_tz,\widehat A)
 \leq D(\Pi_tz,\iota(\theta_t\omega))
      +D(\iota(\theta_t\omega),\iota(a_t))
 <\varepsilon.
\]
Thus $\pi_\Omega^{-1}(U)$ is uniformly attracted to $\widehat A$, and $\widehat A$ is a skew-product attractor.

Conversely, let $C$ be an attractor of $\Pi$ with attracting neighborhood $N$. Theorem~\ref{thm:compact-invariant-correspondence} gives $C=\iota(A)$, where $A=\pi_\Omega(C)$ is compact and strictly invariant. The set
\begin{equation}
 U:=\iota^{-1}(N\cap\Gamma_h)
 \label{eq:base-attracting-neighborhood-from-graph}
\end{equation}
is an open neighborhood of $A$ in $\Omega$. Fix $\varepsilon>0$. Uniform continuity of $\pi_\Omega|_{\Gamma_h}$ provides $\rho>0$ such that
\[
 D(z_1,z_2)<\rho,\quad z_1,z_2\in\Gamma_h
 \quad\Longrightarrow\quad
 d_\Omega(\pi_\Omega z_1,\pi_\Omega z_2)<\varepsilon.
\]
Uniform attraction of $N$ to $C$ gives $T>0$ such that
\[
 \dist(\Pi_tz,C)<\rho,
 \qquad z\in N,\quad t\geq T.
\]
For $\omega\in U$, the point $\iota(\omega)$ lies in $N\cap\Gamma_h$, and graph equivariance gives $\Pi_t\iota(\omega)=\iota(\theta_t\omega)\in\Gamma_h$. Hence, for $t\geq T$, one can choose $a_t\in A$ with
\[
 D(\iota(\theta_t\omega),\iota(a_t))<\rho,
\]
and therefore $d_\Omega(\theta_t\omega,a_t)<\varepsilon$. This proves uniform attraction of $U$ to $A$, so $A$ is a base attractor. The constructions are mutually inverse and preserve inclusions.
\end{pf}

\begin{theorem}
\label{thm:attractor-repeller-correspondence}
Let $A$ be an attractor of the base, let $A_\theta^*$ be its standard dual repeller, and let $\widehat A=\iota(A)$. Then
\begin{equation}
  (\widehat A)_\Pi^*
 =\pi_\Omega^{-1}(A_\theta^*).
 \label{eq:dual-repeller-full-fibers}
\end{equation}
The largest compact strictly invariant subset of this full-fiber repeller is
\begin{equation}
 \Core_\Pi^{\mathrm{str}}\bigl((\widehat A)_\Pi^*\bigr)
 =\iota(A_\theta^*).
 \label{eq:dual-repeller-strict-core}
\end{equation}
Consequently,
\begin{equation}
 (A,A_\theta^*)
 \longmapsto
 \left(\iota(A),\pi_\Omega^{-1}(A_\theta^*)\right)
 \label{eq:attractor-repeller-pair-map}
\end{equation}
is a bijection between the attractor--repeller pairs associated with base attractors and those associated with skew-product attractors. For associated attractor--repeller pairs, define
\begin{equation}
 (A,R)\leq(B,S)
 \quad\Longleftrightarrow\quad
 A\subseteq B,
 \label{eq:attractor-repeller-pair-order}
\end{equation}
or, equivalently, $S\subseteq R$. Under this order, the correspondence \eqref{eq:attractor-repeller-pair-map} is an order isomorphism. Since the attractor posets carry their standard bounded distributive lattice structures, the order isomorphism $A\mapsto\iota(A)$ is automatically a lattice isomorphism. The associated standard dual-repeller map is $R\mapsto\pi_\Omega^{-1}(R)$, whereas $R\mapsto\iota(R)$ describes the largest compact strictly invariant cores.
\end{theorem}

\begin{pf}
Let $z=(\omega,x)\in Z$. By Proposition~\ref{prop:omega-limit-lifting},
\begin{equation}
 z\in(\widehat A)_\Pi^*
 \Longleftrightarrow
 \omega_\Pi(z)\cap\iota(A)=\varnothing
 \Longleftrightarrow
 \iota\bigl(\omega_\theta(\omega)\bigr)\cap\iota(A)=\varnothing
 \Longleftrightarrow
 \omega_\theta(\omega)\cap A=\varnothing
 \Longleftrightarrow
 \omega\in A_\theta^*.
 \label{eq:dual-repeller-proof-equivalences}
\end{equation}
The condition depends only on the base coordinate. Hence every fiber above $A_\theta^*$ belongs to the standard dual repeller, and no point above $\Omega\setminus A_\theta^*$ does. This proves
\[
 (\widehat A)_\Pi^*=\pi_\Omega^{-1}(A_\theta^*).
\]

We next identify the strict invariant core. Since $A_\theta^*$ is compact and strictly invariant under the base flow, equivariance of $\iota$ shows that $\iota(A_\theta^*)$ is compact and strictly invariant under $\Pi$. It is contained in the full-fiber repeller because its base projection lies in $A_\theta^*$.

Let now $C\subset(\widehat A)_\Pi^*$ be any compact strictly invariant set. Theorem~\ref{thm:compact-invariant-correspondence} gives a compact strictly invariant base set $D$ such that $C=\iota(D)$. The inclusion $C\subset\pi_\Omega^{-1}(A_\theta^*)$ implies
\[
 D=\pi_\Omega(C)\subset A_\theta^*.
\]
Consequently $C=\iota(D)\subset\iota(A_\theta^*)$. Thus every compact strictly invariant subset of the repeller is contained in $\iota(A_\theta^*)$, which proves that this graph lift is the largest such subset and establishes \eqref{eq:dual-repeller-strict-core}.

The attractor correspondence is bijective by Theorem~\ref{thm:attractor-correspondence}, and \eqref{eq:dual-repeller-full-fibers} uniquely determines the dual repeller associated with each lifted attractor. If $A\subseteq B$, attractor--repeller duality in the base gives $B_\theta^*\subseteq A_\theta^*$, and therefore
\begin{equation}
 \pi_\Omega^{-1}(B_\theta^*)
 \subseteq\pi_\Omega^{-1}(A_\theta^*).
 \label{eq:dual-repeller-order-reversal}
\end{equation}
The map in \eqref{eq:attractor-repeller-pair-map} therefore preserves and reflects the order defined by the attractor coordinate, equivalently the reverse order on the repeller coordinate.
\end{pf}

\begin{remark}
\label{rem:attractor-repeller-asymmetry}
Contraction collapses attractors and all compact strictly invariant sets to the graph, but it does not collapse the standard dual repeller of a skew-product attractor. Every fiber point above the base dual repeller has an omega-limit set disjoint from the attractor, even though only the graph point in that fiber belongs to the strict invariant core. This is the invariant-set counterpart of the source--target asymmetry in Theorem~\ref{thm:abstract-chain-reduction}.
\end{remark}

\subsection{Morse, invariant-measure, and entropy consequences}
\label{subsec:known-consequences}

The invariant graph also transports finite Morse decompositions, invariant measures, and entropy. The Morse-decomposition statement is a specialization of the skew-product theory in \cite{BortolanEtAl2013}, while the measure and entropy conclusions fit the shrinking-fiber framework of \cite{Kloeckner2021}.

\begin{corollary}
\label{cor:morse-lifting}
If $\{M_1,\ldots,M_r\}$ is a finite Morse decomposition of the base, then
\begin{equation}
 \{\iota(M_1),\ldots,\iota(M_r)\}
 \label{eq:lifted-morse-decomposition}
\end{equation}
is a finite Morse decomposition of the global attractor $\Gamma_h$, with the same order. If the base Morse sets are its chain components, the lifted Morse sets are precisely the chain components of $\Pi$, and their order is the one in \eqref{eq:component-order-isomorphism}.
\end{corollary}

\begin{pf}
The restriction $\Pi|_{\Gamma_h}$ is conjugate to $\theta$, and $\Gamma_h$ is the global attractor. Conjugacy transports the Morse sets and connecting orbits. Preservation of the component order follows from Theorem~\ref{thm:chain-component-correspondence}.
\end{pf}

\begin{corollary}
\label{cor:entropy-identity}
Let $T:=\theta_\tau$ and $S:=\Pi_\tau$. Every $S$-invariant Borel probability measure is supported on $\Gamma_h$ and is uniquely of the form $\iota_*\mu$ for a $T$-invariant probability measure $\mu$. Consequently,
\begin{equation}
 h_{\mathrm{top}}(S)=h_{\mathrm{top}}(T).
 \label{eq:entropy-identity}
\end{equation}
\end{corollary}

\begin{pf}
For an $S$-invariant probability measure $\nu$, \eqref{eq:one-step-V-contraction} and invariance give
\begin{equation}
 \int_ZV\,\dd\nu
 =\int_ZV\circ S\,\dd\nu
 \leq q\int_ZV\,\dd\nu.
 \label{eq:invariant-measure-on-graph}
\end{equation}
Since $0<q<1$ and $V\geq0$, one has $\int V\,\dd\nu=0$, hence $\nu(\Gamma_h)=1$. The graph conjugacy gives the unique representation $\nu=\iota_*\mu$. Measure-theoretic entropy is preserved under this measure conjugacy, and the variational principle \cite{Walters1982} yields \eqref{eq:entropy-identity}. This is the strict-contraction specialization of the broader shrinking-fiber conclusions in \cite{Kloeckner2021}.
\end{pf}

\section{Subadditive fiber contraction}
\label{sec:subadditive-contraction}

Finite-window contraction is the geometric hypothesis used in Sections~\ref{sec:graph} and \ref{sec:recurrent}. We now derive it from the asymptotic growth of the fiber derivative cocycle. The regularity required by the continuous-time subadditive variational theorem is stated explicitly.

\subsection{The fiber derivative-growth cocycle}
\label{subsec:derivative-growth}

For a linear map $A:\R^n\to\R^n$, define the operator norm between moving fiber norms by
\begin{equation}
 \lVert A\rVert_{\omega\to\eta}
 :=\sup_{v\neq0}\frac{\lvert Av\rvert_\eta}{\lvert v\rvert_\omega}.
 \label{eq:moving-operator-norm}
\end{equation}
The regularity assumptions on $F$ imply that $x\mapsto\varphi(t,\omega,x)$ is continuously differentiable. Set
\begin{equation}
 \Lambda_t(\omega)
 :=\sup_{x\in K}
 \left\lVert D_x\varphi(t,\omega,x)
 \right\rVert_{\omega\to\theta_t\omega},
 \qquad
 \kappa_t(\omega):=\log\Lambda_t(\omega).
 \label{eq:fiber-growth-cocycle}
\end{equation}

\begin{lemma}
\label{lem:joint-continuity-growth}
For every $T>0$, the map
\begin{equation}
 (t,\omega,x)
 \longmapsto D_x\varphi(t,\omega,x)
 \label{eq:joint-derivative-flow-map}
\end{equation}
is continuous on $[0,T]\times\Omega\times K$. Consequently,
\begin{equation}
 (t,\omega)\longmapsto\Lambda_t(\omega)
 \quad\text{and}\quad
 (t,\omega)\longmapsto\kappa_t(\omega)
 \label{eq:joint-continuity-Lambda-kappa}
\end{equation}
are continuous on $[0,T]\times\Omega$. Moreover, $\Lambda_t(\omega)>0$ for all $t\geq0$ and $\omega\in\Omega$.
\end{lemma}

\begin{pf}
For fixed $(\omega,x)$, the derivative matrix
\begin{equation}
 X(t,\omega,x):=D_x\varphi(t,\omega,x)
 \label{eq:variational-fundamental-matrix}
\end{equation}
solves
\begin{equation}
 \dot X(t)
 =D_xF\bigl(\theta_t\omega,\varphi(t,\omega,x)\bigr)X(t),
 \qquad X(0)=I.
 \label{eq:variational-equation-joint}
\end{equation}
The coefficient matrix is jointly continuous in $(t,\omega,x)$ on every compact time interval because $D_xF$, $\theta$, and $\varphi$ are continuous. Continuous dependence for linear differential equations therefore gives joint continuity of $X$ in \eqref{eq:joint-derivative-flow-map}.

The moving operator norm is continuous in the matrix and in the two base points because $P$ is continuous and uniformly positive definite. Hence
\begin{equation}
 (t,\omega,x)
 \longmapsto
 \left\lVert D_x\varphi(t,\omega,x)
 \right\rVert_{\omega\to\theta_t\omega}
 \label{eq:joint-continuous-integrand-Lambda}
\end{equation}
is continuous on $[0,T]\times\Omega\times K$. Since $K$ is compact, the maximum theorem yields joint continuity of $\Lambda_t(\omega)$. The fundamental matrix of \eqref{eq:variational-equation-joint} is nonsingular, so its operator norm is positive; therefore $\Lambda_t(\omega)>0$. Continuity of the logarithm then gives joint continuity of $\kappa_t(\omega)$.
\end{pf}

\begin{lemma}
\label{lem:subadditive-growth}
For all $s,t\geq0$ and $\omega\in\Omega$,
\begin{equation}
 \kappa_{t+s}(\omega)
 \leq\kappa_s(\omega)+\kappa_t(\theta_s\omega).
 \label{eq:kappa-subadditivity}
\end{equation}
Moreover, for every $x,y\in K$,
\begin{equation}
 \left\lvert
 \varphi(t,\omega,x)-\varphi(t,\omega,y)
 \right\rvert_{\theta_t\omega}
 \leq\exp\bigl(\kappa_t(\omega)\bigr)
 \lvert x-y\rvert_\omega.
 \label{eq:kappa-nonlinear-estimate}
\end{equation}
\end{lemma}

\begin{pf}
Differentiating the cocycle identity with respect to the initial fiber state gives
\begin{equation}
 D_x\varphi(t+s,\omega,x)
 =D_x\varphi\bigl(t,\theta_s\omega,\varphi(s,\omega,x)\bigr)D_x\varphi(s,\omega,x).
 \label{eq:derivative-cocycle}
\end{equation}
The intermediate state belongs to $K$ by positive invariance. Submultiplicativity of the moving operator norms and the supremum over $x\in K$ give
\begin{equation}
 \Lambda_{t+s}(\omega)
 \leq\Lambda_t(\theta_s\omega)\Lambda_s(\omega),
 \label{eq:Lambda-submultiplicativity}
\end{equation}
which is equivalent to \eqref{eq:kappa-subadditivity}. Since $K$ is convex, the segment $y+r(x-y)$, $0\leq r\leq1$, lies in $K$. The mean-value integral formula gives
\begin{equation}
 \varphi(t,\omega,x)-\varphi(t,\omega,y)
 =\int_0^1D_x\varphi\bigl(t,\omega,y+r(x-y)\bigr)(x-y)\,\dd r.
 \label{eq:flow-mean-value-formula}
\end{equation}
Taking the moving norm and using \eqref{eq:fiber-growth-cocycle} proves \eqref{eq:kappa-nonlinear-estimate}.
\end{pf}

\subsection{Invariant-measure exponents and the semi-uniform formula}
\label{subsec:subadditive-semiuniform}

Let $\Mcal_\theta(\Omega)$ and $\Mcal_\theta^{\mathrm e}(\Omega)$ denote, respectively, the invariant and ergodic invariant Borel probability measures of the base flow. For $\mu\in\Mcal_\theta(\Omega)$, invariance and subadditivity imply that
\begin{equation}
 b_\mu(t):=\int_\Omega\kappa_t(\omega)\,\dd\mu(\omega)
 \label{eq:integrated-subadditive-function}
\end{equation}
is subadditive in $t$. Hence the limit
\begin{equation}
 \chi(\mu)
 :=\lim_{t\to\infty}\frac{1}{t}b_\mu(t)
 =\inf_{t>0}\frac{1}{t}
 \int_\Omega\kappa_t(\omega)\,\dd\mu(\omega)
 \label{eq:measure-fiber-exponent}
\end{equation}
exists. Define
\begin{equation}
 \chi_*
 :=\sup_{\mu\in\Mcal_\theta(\Omega)}\chi(\mu).
 \label{eq:maximal-fiber-exponent}
\end{equation}

\begin{lemma}
\label{lem:subadditive-semiuniform}
For the jointly continuous subadditive family $\{\kappa_t\}_{t\geq0}$,
\begin{equation}
 \chi_*
 =\sup_{\mu\in\Mcal_\theta^{\mathrm e}(\Omega)}\chi(\mu)
  =\inf_{t>0}\frac{1}{t}
 \sup_{\omega\in\Omega}\kappa_t(\omega)
 =\lim_{t\to\infty}\frac{1}{t}
 \sup_{\omega\in\Omega}\kappa_t(\omega).
 \label{eq:schreiber-variational-growth}
\end{equation}
Consequently, if $\chi_*<0$, then there exist $\tau>0$ and $\gamma>0$ such that
\begin{equation}
 \kappa_\tau(\omega)\leq-\gamma,
 \qquad \omega\in\Omega.
 \label{eq:subadditive-uniform-window}
\end{equation}
\end{lemma}

\begin{pf}
Schreiber's continuous-time subadditive variational theorem \cite[Theorem~1]{Schreiber1998}, applied using Lemma~\ref{lem:joint-continuity-growth}, gives
\begin{equation}
 r
 :=\sup_{\mu\in\Mcal_\theta^{\mathrm e}(\Omega)}\chi(\mu)
 =\inf_{t>0}\frac{1}{t}
 \sup_{\omega\in\Omega}\kappa_t(\omega).
 \label{eq:schreiber-ergodic-value}
\end{equation}
For an arbitrary invariant probability measure $\mu$,
\begin{equation}
 \chi(\mu)
 =\inf_{t>0}\frac{1}{t}\int_\Omega\kappa_t(\omega)\,\dd\mu(\omega)
 \leq\inf_{t>0}\frac{1}{t}\sup_{\omega\in\Omega}\kappa_t(\omega)
 =r.
 \label{eq:all-invariant-measures-upper-bound}
\end{equation}
Hence $\sup_{\Mcal_\theta(\Omega)}\chi\leq r$. The reverse inequality is immediate from
\begin{equation}
 \Mcal_\theta^{\mathrm e}(\Omega)
 \subset\Mcal_\theta(\Omega),
 \label{eq:ergodic-subset-invariant}
\end{equation}
so the supremum over all invariant measures equals the ergodic supremum without any additional interchange of ergodic decomposition and subadditive limits.

Finally, the function
\begin{equation}
 a(t):=\sup_{\omega\in\Omega}\kappa_t(\omega)
 \label{eq:uniform-growth-function}
\end{equation}
is subadditive and locally bounded by joint continuity. The continuous-time Fekete lemma gives
\begin{equation}
 \lim_{t\to\infty}\frac{a(t)}{t}
 =\inf_{t>0}\frac{a(t)}{t},
 \label{eq:continuous-fekete}
\end{equation}
which proves the last equality in \eqref{eq:schreiber-variational-growth}. If $\chi_*<0$, choose $\tau>0$ such that
\begin{equation}
 \frac{1}{\tau}\sup_{\omega\in\Omega}\kappa_\tau(\omega)<0.
 \label{eq:negative-uniform-growth-time}
\end{equation}
Setting $\gamma:=-\sup_\omega\kappa_\tau(\omega)>0$ proves \eqref{eq:subadditive-uniform-window}.
\end{pf}

\begin{theorem}
	\label{thm:subadditive-main}
	Suppose that Assumption~\ref{ass:well-posedness} holds and that
	\begin{equation}
		\chi_*<0.
		\label{eq:negative-maximal-fiber-exponent}
	\end{equation}
	Then there exist \(\tau>0\) and \(q\in(0,1)\) such that the cocycle is
	uniformly contracting over the window \(\tau\). Consequently:
	\begin{enumerate}
		\item there exists a unique continuous invariant graph \(\Gamma_h\), and
		every fiber trajectory tracks the corresponding graph trajectory
		exponentially;
		
		\item the ambient chain-attainability relation is given by
		Corollary~\ref{cor:controlled-chain-reduction}; in particular,
		\eqref{eq:CR-identity} holds and the chain-component order is transported
		isomorphically from the base;
		
		\item compact strictly invariant sets and attractors are graph lifts, the
		basin identity \eqref{eq:basin-identity} holds, standard dual repellers
		are full inverse images under \(\pi_\Omega\), and their largest compact
		strictly invariant cores are graph lifts;
		
		\item the perturbation and practical-confinement conclusions of
		Section~\ref{sec:robustness} apply whenever their additional hypotheses
		are satisfied.
	\end{enumerate}
	Moreover, the invariant-graph conjugacy and the shrinking-fiber structure
	yield the Morse-lifting and entropy consequences stated in
	Corollaries~\ref{cor:morse-lifting} and
	\ref{cor:entropy-identity}, respectively.
\end{theorem}

\begin{pf}
	By Lemma~\ref{lem:subadditive-semiuniform}, there exist \(\tau>0\) and
	\(\gamma>0\) such that \eqref{eq:subadditive-uniform-window} holds.
	Lemma~\ref{lem:subadditive-growth} then gives
	\begin{equation}
		\left\lvert
		\varphi(\tau,\omega,x)-\varphi(\tau,\omega,y)
		\right\rvert_{\theta_\tau\omega}
		\leq
		\e^{-\gamma}\lvert x-y\rvert_\omega.
		\label{eq:subadditive-implies-block}
	\end{equation}
	Hence Definition~\ref{def:block-contraction} is satisfied with
	\(q=\e^{-\gamma}\in(0,1)\).
	
	Item~(1) follows from the invariant-graph and exponential-tracking results
	of Section~\ref{sec:graph}. Items~(2) and~(3) follow from the chain,
	invariant-set, and attractor--repeller results of Section~\ref{sec:recurrent}.
	Item~(4) follows from Section~\ref{sec:robustness} once the corresponding
	additional perturbation hypotheses are imposed. The final assertion is
	precisely the content of Corollaries~\ref{cor:morse-lifting} and
	\ref{cor:entropy-identity}.
\end{pf}

\begin{remark}
\label{rem:subadditive-strength}
Theorem~\ref{thm:subadditive-main} does not require an additive pointwise upper bound for $\kappa_t$. It applies whenever the nonlinear fiber derivative cocycle has negative maximal invariant-measure exponent, even if no useful instantaneous logarithmic-norm estimate is available. The moving-metric condition in the next section is a computable sufficient certificate for this more intrinsic hypothesis.
\end{remark}

\section{Averaged moving-metric certificates}
\label{sec:averaged-contraction}

We next derive a verifiable differential sufficient condition for Theorem~\ref{thm:subadditive-main}. The argument is aligned with classical contraction analysis, differential Lyapunov methods, and control contraction metrics \cite{LohmillerSlotine1998,ForniSepulchre2014,ManchesterSlotine2017}. The moving logarithmic norm may be positive on subsets of the base, but its average must be negative for every invariant probability measure.

\subsection{An additive semi-uniform averaging lemma}
\label{subsec:semiuniform}

Denote by \(\Mcal_\theta(\Omega)\) the compact convex set of \(\theta\)-invariant Borel probability measures on \(\Omega\).

\begin{lemma}
\label{lem:semiuniform-negativity}
Let \(a\in C(\Omega,\R)\) and define
\begin{equation}
 \beta(a)
 :=\sup_{\mu\in\Mcal_\theta(\Omega)}
 \int_\Omega a(\omega)\,\dd\mu(\omega).
 \label{eq:maximal-ergodic-average}
\end{equation}
If \(\beta(a)<0\), then there exist \(\tau>0\) and \(\gamma>0\) such that
\begin{equation}
 \int_0^\tau a(\theta_s\omega)\,\dd s
 \leq-\gamma
 \label{eq:uniform-negative-window}
\end{equation}
for every \(\omega\in\Omega\). More precisely, for every \(\varepsilon>0\) there exists \(T_\varepsilon>0\) such that
\begin{equation}
 \sup_{\omega\in\Omega}
 \frac{1}{T}\int_0^T a(\theta_s\omega)\,\dd s
 \leq\beta(a)+\varepsilon,
 \qquad T\geq T_\varepsilon.
 \label{eq:semiuniform-average-bound}
\end{equation}
\end{lemma}

\begin{pf}
Suppose that \eqref{eq:semiuniform-average-bound} fails for some \(\varepsilon>0\). Then there exist \(T_j\to\infty\) and \(\omega_j\in\Omega\) such that
\begin{equation}
 \frac{1}{T_j}\int_0^{T_j}a(\theta_s\omega_j)\,\dd s
 >\beta(a)+\varepsilon.
 \label{eq:semiuniform-contradiction-sequence}
\end{equation}
The occupation measures
\begin{equation}
 \mu_j:=\frac{1}{T_j}\int_0^{T_j}\delta_{\theta_s\omega_j}\,\dd s
 \label{eq:occupation-measures}
\end{equation}
possess a weakly convergent subsequence. Its limit is \(\theta\)-invariant, while \eqref{eq:semiuniform-contradiction-sequence} contradicts the definition of \(\beta(a)\). Choosing \(\varepsilon=-\beta(a)/2\) gives \eqref{eq:uniform-negative-window}.
\end{pf}

\subsection{Moving logarithmic norm}
\label{subsec:moving-lognorm}

Assume that \(P\) is differentiable along the base flow and that
\begin{equation}
 \dot P(\omega)
 :=\left.\frac{\dd}{\dd t}P(\theta_t\omega)\right|_{t=0}
 \label{eq:P-derivative-along-base}
\end{equation}
exists and depends continuously on \(\omega\). Define
\begin{equation}
 \mathcal{S}_P(\omega,x)
 :=D_xF(\omega,x)^\top P(\omega)
 +P(\omega)D_xF(\omega,x)
 +\dot P(\omega)
 \label{eq:moving-symmetric-part}
\end{equation}
and
\begin{equation}
 \nu_P(\omega)
 :=\frac{1}{2}\sup_{x\in K}
 \lambdaMax\!\left(
 P(\omega)^{-1/2}\mathcal{S}_P(\omega,x)P(\omega)^{-1/2}
 \right).
 \label{eq:moving-logarithmic-norm}
\end{equation}

\begin{proposition}
\label{prop:differential-contraction}
For all \(t\geq0\), \(\omega\in\Omega\), and \(x,y\in K\),
\begin{equation}
 \left\lvert
 \varphi(t,\omega,x)-\varphi(t,\omega,y)
 \right\rvert_{\theta_t\omega}
 \leq
 \exp\!\left(\int_0^t\nu_P(\theta_s\omega)\,\dd s\right)
 \lvert x-y\rvert_\omega.
 \label{eq:differential-incremental-estimate}
\end{equation}
Consequently,
\begin{equation}
 \kappa_t(\omega)
 \leq\int_0^t\nu_P(\theta_s\omega)\,\dd s.
 \label{eq:kappa-lognorm-bound}
\end{equation}
\end{proposition}

\begin{pf}
Let \(x_1(t)=\varphi(t,\omega,x)\), \(x_2(t)=\varphi(t,\omega,y)\), and \(\xi(t)=x_1(t)-x_2(t)\). Since \(K\) is convex and positively invariant,
\begin{equation}
 \dot\xi(t)=A(t)\xi(t),
 \qquad
 A(t)=\int_0^1D_xF\bigl(\theta_t\omega,x_2(t)+r\xi(t)\bigr)\,\dd r.
 \label{eq:mean-value-variational-matrix}
\end{equation}
By the definition of \(\nu_P\), for every \(z\in K\),
\begin{equation}
 D_xF(\theta_t\omega,z)^\top P(\theta_t\omega)
 +P(\theta_t\omega)D_xF(\theta_t\omega,z)
 +\dot P(\theta_t\omega)
 \preceq2\nu_P(\theta_t\omega)P(\theta_t\omega).
 \label{eq:pointwise-lognorm-matrix-bound}
\end{equation}
The left-hand side is affine in the Jacobian, so integration with respect to \(r\) yields the same inequality with \(D_xF\) replaced by \(A(t)\). Therefore, for
\begin{equation}
 W(t):=\xi(t)^\top P(\theta_t\omega)\xi(t),
 \label{eq:incremental-energy}
\end{equation}
one has
\begin{equation}
 \dot W(t)
 =\xi(t)^\top
 \bigl(A(t)^\top P(\theta_t\omega)
 +P(\theta_t\omega)A(t)
 +\dot P(\theta_t\omega)\bigr)\xi(t)
 \leq2\nu_P(\theta_t\omega)W(t).
 \label{eq:incremental-energy-inequality}
\end{equation}
Integration gives
\begin{equation}
 W(t)\leq W(0)
 \exp\!\left(2\int_0^t\nu_P(\theta_s\omega)\,\dd s\right),
 \label{eq:incremental-energy-gronwall}
\end{equation}
and taking square roots proves \eqref{eq:differential-incremental-estimate}. Applying the same computation to a variational vector
\begin{equation}
 \dot v=D_xF(\theta_t\omega,\varphi(t,\omega,x))v
 \label{eq:fiber-variational-equation}
\end{equation}
and then taking the operator norm and the supremum over \(x\in K\) proves \eqref{eq:kappa-lognorm-bound}.
\end{pf}

Define the maximal invariant-measure logarithmic-norm average
\begin{equation}
 \beta_P
 :=\sup_{\mu\in\Mcal_\theta(\Omega)}
 \int_\Omega\nu_P(\omega)\,\dd\mu(\omega).
 \label{eq:maximal-lognorm-average}
\end{equation}

\begin{theorem}
\label{thm:averaged-contraction-main}
Suppose that Assumption~\ref{ass:well-posedness} holds and that there exists a moving metric \(P\) satisfying \eqref{eq:P-uniform-bounds} and \eqref{eq:P-derivative-along-base}. If
\begin{equation}
 \beta_P<0,
 \label{eq:negative-maximal-lognorm-average}
\end{equation}
then all conclusions of Theorem~\ref{thm:subadditive-main} hold.
\end{theorem}

\begin{pf}
Lemma~\ref{lem:semiuniform-negativity} applied to \(a=\nu_P\) gives \(\tau>0\) and \(\gamma>0\) such that
\begin{equation}
 \int_0^\tau\nu_P(\theta_s\omega)\,\dd s\leq-\gamma
 \label{eq:lognorm-negative-window}
\end{equation}
for every \(\omega\in\Omega\). Proposition~\ref{prop:differential-contraction} yields finite-window contraction with \(q=\e^{-\gamma}\). Alternatively, \eqref{eq:kappa-lognorm-bound} gives \(\chi_*\leq\beta_P<0\), so Theorem~\ref{thm:subadditive-main} applies directly.
\end{pf}

\subsection{Control-affine certificate}
\label{subsec:control-affine-certificate}

For the control-affine system
\begin{equation}
 \dot x=f(\theta_t\omega,x)+B(\theta_t\omega,x)u,
 \label{eq:control-affine-system}
\end{equation}
let \(u=k(\omega,x)\), and set
\begin{equation}
 A_k(\omega,x)
 :=D_x\bigl(f(\omega,x)+B(\omega,x)k(\omega,x)\bigr).
 \label{eq:closed-loop-jacobian}
\end{equation}

\begin{corollary}
\label{cor:control-certificate}
Assume there exist a feedback \(k\), a moving metric \(P\), and a continuous function \(a:\Omega\to\R\) such that
\begin{equation}
 A_k(\omega,x)^\top P(\omega)
 +P(\omega)A_k(\omega,x)
 +\dot P(\omega)
 \preceq2a(\omega)P(\omega)
 \label{eq:feedback-metric-inequality}
\end{equation}
for every \((\omega,x)\in\Omega\times K\), and
\begin{equation}
 \sup_{\mu\in\Mcal_\theta(\Omega)}
 \int_\Omega a(\omega)\,\dd\mu(\omega)<0.
 \label{eq:feedback-average-negativity}
\end{equation}
Then all conclusions of Theorem~\ref{thm:subadditive-main} hold for the closed-loop skew product.
\end{corollary}

\begin{pf}
Inequality \eqref{eq:feedback-metric-inequality} implies \(\nu_P(\omega)\leq a(\omega)\), so \(\beta_P<0\). Apply Theorem~\ref{thm:averaged-contraction-main}.
\end{pf}

\begin{remark}
\label{rem:state-dependent-metric}
A Riemannian version can be formulated with a state-dependent metric \(P(\omega,x)\), the total derivative of the metric along the closed-loop vector field, and the induced geodesic distance. The compact-convex setting used here keeps the main mechanism transparent and avoids additional completeness and geodesic-convexity assumptions. Sections~\ref{sec:graph} and \ref{sec:recurrent} require only the resulting finite-window metric contraction.
\end{remark}

\subsection{Uniquely ergodic and periodic bases}
\label{subsec:unique-ergodicity}

The invariant-measure criterion simplifies substantially under unique
ergodicity and reduces to an explicit time-average condition for periodic
base orbits. We first state the uniquely ergodic case and then record its
periodic and quasiperiodic specializations.

\begin{corollary}
\label{cor:uniquely-ergodic}
Suppose that \(\theta\) is uniquely ergodic with invariant probability measure \(\mu_*\). Under the differential inequality \eqref{eq:feedback-metric-inequality}, it is sufficient that
\begin{equation}
 \int_\Omega a(\omega)\,\dd\mu_*(\omega)<0.
 \label{eq:unique-ergodic-condition}
\end{equation}
\end{corollary}

\begin{pf}
Unique ergodicity implies that \(\Mcal_\theta(\Omega)=\{\mu_*\}\), so \eqref{eq:unique-ergodic-condition} is exactly \eqref{eq:feedback-average-negativity}.
\end{pf}

For a periodic base orbit of period \(T_p\), this becomes
\begin{equation}
 \frac{1}{T_p}\int_0^{T_p}a(\theta_t\omega)\,\dd t<0.
 \label{eq:period-average-condition}
\end{equation}
For an irrational translation on \(\T^m\), \(\mu_*\) is Haar measure, so the test is an ordinary torus average.

\section{Robustness of the invariant and recurrent graphs}
\label{sec:robustness}

We compare two controlled cocycles \(\varphi\) and \(\widetilde\varphi\) over
the same base flow. Unless stated otherwise, both are assumed to leave the
same compact fiber set \(K\) positively invariant. The first theorem is
formulated directly for their time-\(\tau\) maps. We then derive a genuinely
\(C^1\) estimate through a modulus of continuity of the Jacobian, followed by
a simpler explicit \(C^{1,1}\) specialization, where \(C^{1,1}\) means that
the fiber Jacobian is Lipschitz continuous with respect to the state variable.

\subsection{Persistence of a common invariant fiber set}
\label{subsec:robust-invariance-K}

The use of a common compact invariant set is a genuine hypothesis and is not implied by $C^1$ closeness alone. The following elementary condition makes this assumption verifiable.

\begin{proposition}
\label{prop:robust-invariance-K}
Let
\begin{equation}
 K=\{x\in\R^n:g(x)\leq0\}
 \label{eq:K-sublevel-description}
\end{equation}
be compact, where $g\in C^1$ on a neighborhood of $K$ and $\nabla g(x)\neq0$ for $x\in\partial K$. Suppose that the nominal field satisfies the uniform inward-pointing margin
\begin{equation}
 \nabla g(x)^\top F(\omega,x)\leq-\sigma<0,
 \qquad (\omega,x)\in\Omega\times\partial K.
 \label{eq:nominal-inward-margin}
\end{equation}
If
\begin{equation}
 \sup_{\Omega\times\partial K}
 \lVert \widetilde F-F\rVert
 <\frac{\sigma}{\sup_{x\in\partial K}\lVert\nabla g(x)\rVert},
 \label{eq:inward-perturbation-bound}
\end{equation}
then $K$ is positively invariant for the perturbed system as well.
\end{proposition}

\begin{pf}
For $(\omega,x)\in\Omega\times\partial K$,
\begin{equation}
 \nabla g(x)^\top\widetilde F(\omega,x)
 \leq\nabla g(x)^\top F(\omega,x)
 +\lVert\nabla g(x)\rVert\lVert\widetilde F(\omega,x)-F(\omega,x)\rVert<0.
 \label{eq:perturbed-inward-inequality}
\end{equation}
A standard first-exit argument then excludes crossing of $\partial K$ from the interior to the exterior.
\end{pf}

\subsection{Persistence under perturbations of the time-window map}
\label{subsec:graph-perturbation}

Assume that \(\varphi\) contracts over \(\tau\) with factor \(q<1\). Define
\begin{equation}
 \eta_\tau
 :=\sup_{\omega\in\Omega,\,x\in K}
 \left\lvert
 \varphi(\tau,\omega,x)-\widetilde\varphi(\tau,\omega,x)
 \right\rvert_{\theta_\tau\omega}
 \label{eq:one-window-mismatch}
\end{equation}
and
\begin{equation}
 \delta_\tau
 :=\sup_{\substack{\omega\in\Omega\\x,y\in K,\ x\neq y}}
 \frac{\left\lvert
 [\widetilde\varphi(\tau,\omega,x)-\varphi(\tau,\omega,x)]
 -[\widetilde\varphi(\tau,\omega,y)-\varphi(\tau,\omega,y)]
 \right\rvert_{\theta_\tau\omega}}
 {\lvert x-y\rvert_\omega}.
 \label{eq:lipschitz-window-mismatch}
\end{equation}

\begin{theorem}
\label{thm:graph-robustness}
If
\begin{equation}
 \delta_\tau<1-q,
 \label{eq:robust-contraction-gap}
\end{equation}
then \(\widetilde\varphi\) contracts over \(\tau\) with factor
\begin{equation}
 \widetilde q:=q+\delta_\tau<1.
 \label{eq:perturbed-contraction-factor}
\end{equation}
Consequently, it possesses a unique continuous invariant graph \(\widetilde h\). If \(h\) is the nominal invariant graph, then
\begin{equation}
 \sup_{\omega\in\Omega}
 \lvert h(\omega)-\widetilde h(\omega)\rvert_\omega
 \leq\frac{\eta_\tau}{1-q}.
 \label{eq:graph-robustness-bound}
\end{equation}
\end{theorem}

\begin{pf}
Write
\begin{equation}
 E_\omega(x)
 :=\widetilde\varphi(\tau,\omega,x)-\varphi(\tau,\omega,x).
 \label{eq:window-error-map}
\end{equation}
Then, for \(x,y\in K\),
\begin{equation}
\begin{aligned}
 &\left\lvert
 \widetilde\varphi(\tau,\omega,x)
 -\widetilde\varphi(\tau,\omega,y)
 \right\rvert_{\theta_\tau\omega}
 \\[-1pt]
 &\qquad\leq
 \left\lvert
 \varphi(\tau,\omega,x)-\varphi(\tau,\omega,y)
 \right\rvert_{\theta_\tau\omega}
 \\[-1pt]
 &\qquad\quad+\lvert E_\omega(x)-E_\omega(y)\rvert_{\theta_\tau\omega}
 \\[-1pt]
 &\qquad\leq(q+\delta_\tau)\lvert x-y\rvert_\omega.
\end{aligned}
 \label{eq:perturbed-window-contraction}
\end{equation}
Thus the perturbed graph transform is a contraction with factor \(\widetilde q\).

Let \(\Gcal\) and \(\widetilde\Gcal\) be the nominal and perturbed graph transforms. Since \(h=\Gcal h\) and \(\widetilde h=\widetilde\Gcal\widetilde h\),
\begin{equation}
 d_\infty(h,\widetilde h)
 \leq d_\infty(\Gcal h,\Gcal\widetilde h)
 +d_\infty(\Gcal\widetilde h,
 \widetilde\Gcal\widetilde h)
 \leq qd_\infty(h,\widetilde h)+\eta_\tau.
 \label{eq:graph-robustness-proof}
\end{equation}
Rearranging proves \eqref{eq:graph-robustness-bound}. Notice that the denominator contains the nominal gap \(1-q\); the existence of the perturbed graph is ensured separately by \(q+\delta_\tau<1\).
\end{pf}

\subsection{Finite-window estimates from vector-field perturbations}
\label{subsec:vector-field-perturbations}

Let $F$ and $\widetilde F$ be the nominal and perturbed closed-loop vector fields. Assume that both are defined on a common neighborhood of $K$, their solutions remain in $K$ over $[0,\tau]$, and
\begin{equation}
 L:=\max\left\{
 \sup_{\Omega\times K}\lVert D_xF\rVert,
 \sup_{\Omega\times K}\lVert D_x\widetilde F\rVert
 \right\}<\infty.
 \label{eq:robust-L}
\end{equation}
Set
\begin{equation}
 \varepsilon_0
 :=\sup_{\Omega\times K}\lVert F-\widetilde F\rVert,
 \qquad
 \varepsilon_1
 :=\sup_{\Omega\times K}\lVert D_xF-D_x\widetilde F\rVert.
 \label{eq:C1-vector-field-errors}
\end{equation}
Define
\begin{equation}
 c_0(t):=
 \begin{cases}
 \dfrac{\e^{Lt}-1}{L},&L>0,\\[2mm]
 t,&L=0,
 \end{cases}
 \label{eq:c0-definition}
\end{equation}
and the modulus of continuity of the nominal Jacobian on the compact set $\Omega\times K$ by
\begin{equation}
 \omega_{DF}(r)
 :=\sup\left\{
 \lVert D_xF(\omega,x)-D_xF(\omega,y)\rVert:
 \omega\in\Omega,\ x,y\in K,\lVert x-y\rVert\leq r
 \right\}.
 \label{eq:DF-modulus-continuity}
\end{equation}
Then $\omega_{DF}(r)\to0$ as $r\downarrow0$.

\begin{proposition}
\label{prop:C1-to-window-bounds}
Under \eqref{eq:robust-L} and \eqref{eq:C1-vector-field-errors},
\begin{equation}
 \eta_\tau
 \leq\sqrt{M_P}\,c_0(\tau)\varepsilon_0,
 \label{eq:eta-vector-field-bound}
\end{equation}
and
\begin{equation}
 \delta_\tau
 \leq
 \sqrt{\frac{M_P}{m_P}}\,
 \e^{L\tau}
 \int_0^\tau
 \left[
 \varepsilon_1+
 \omega_{DF}\bigl(\varepsilon_0c_0(s)\bigr)
 \right]\,\dd s.
 \label{eq:delta-vector-field-bound-general-C1}
\end{equation}
In particular, the condition
\begin{equation}
 \sqrt{\frac{M_P}{m_P}}\,
 \e^{L\tau}
 \int_0^\tau
 \left[
 \varepsilon_1+
 \omega_{DF}\bigl(\varepsilon_0c_0(s)\bigr)
 \right]\,\dd s
 <1-q
 \label{eq:C1-robustness-condition-general}
\end{equation}
ensures persistence of finite-window contraction, the invariant graph, and the ambient chain-relation classification.
\end{proposition}

\begin{pf}
Let $x(t)$ and $\widetilde x(t)$ solve the two systems from the same initial state. Gr\"onwall's inequality gives
\begin{equation}
 \lVert\widetilde x(t)-x(t)\rVert
 \leq\varepsilon_0c_0(t),
 \qquad 0\leq t\leq\tau.
 \label{eq:state-flow-perturbation-bound}
\end{equation}
Multiplication by $\sqrt{M_P}$ at the terminal base point proves \eqref{eq:eta-vector-field-bound}.

Let
\begin{equation}
 X(t):=D_x\varphi(t,\omega,x_0),
 \qquad
 \widetilde X(t):=D_x\widetilde\varphi(t,\omega,x_0).
 \label{eq:two-variational-flows}
\end{equation}
They satisfy variational equations with coefficient matrices
\begin{equation}
 A(t)=D_xF(\theta_t\omega,x(t)),
 \qquad
 \widetilde A(t)
 =D_x\widetilde F(\theta_t\omega,\widetilde x(t)).
 \label{eq:two-variational-matrices}
\end{equation}
By adding and subtracting $D_xF(\theta_t\omega,\widetilde x(t))$ and using \eqref{eq:state-flow-perturbation-bound},
\begin{equation}
\begin{aligned}
 \lVert\widetilde A(t)-A(t)\rVert
 &\leq
 \lVert D_x\widetilde F(\theta_t\omega,\widetilde x(t))
       -D_xF(\theta_t\omega,\widetilde x(t))\rVert
 \\
 &\quad+
 \lVert D_xF(\theta_t\omega,\widetilde x(t))
       -D_xF(\theta_t\omega,x(t))\rVert
 \\
 &\leq
 \varepsilon_1+
 \omega_{DF}\bigl(\varepsilon_0c_0(t)\bigr).
\end{aligned}
 \label{eq:coefficient-matrix-perturbation-general}
\end{equation}
Moreover, $\lVert X(t)\rVert,\lVert\widetilde X(t)\rVert\leq\e^{Lt}$. Variation of constants for $\widetilde X-X$ gives
\begin{equation}
 \lVert\widetilde X(t)-X(t)\rVert
 \leq\e^{Lt}
 \int_0^t
 \left[
 \varepsilon_1+
 \omega_{DF}\bigl(\varepsilon_0c_0(s)\bigr)
 \right]\,\dd s.
 \label{eq:variational-flow-perturbation-general}
\end{equation}
The mean-value formula for the error map
\begin{equation}
 x\longmapsto
 \widetilde\varphi(\tau,\omega,x)-\varphi(\tau,\omega,x)
 \label{eq:error-map-mean-value}
\end{equation}
therefore bounds its Euclidean Lipschitz constant by the right-hand side of \eqref{eq:variational-flow-perturbation-general} at $t=\tau$. Converting the input and output norms using \eqref{eq:P-uniform-bounds} proves \eqref{eq:delta-vector-field-bound-general-C1}.
\end{pf}

\begin{corollary}
\label{cor:C11-window-bounds}
Assume additionally that
\begin{equation}
 H:=\Lip_x(D_xF)<\infty.
 \label{eq:robust-H}
\end{equation}
Define
\begin{equation}
 \psi_L(t):=
 \begin{cases}
 \dfrac{\e^{Lt}-1-Lt}{L^2},&L>0,\\[2mm]
 \dfrac{t^2}{2},&L=0.
 \end{cases}
 \label{eq:psiL-definition}
\end{equation}
Then
\begin{equation}
 \delta_\tau
 \leq
 \sqrt{\frac{M_P}{m_P}}\,
 \e^{L\tau}
 \left[
 \tau\varepsilon_1
 +H\varepsilon_0\psi_L(\tau)
 \right].
 \label{eq:delta-vector-field-bound}
\end{equation}
Thus the explicit sufficient condition
\begin{equation}
 \sqrt{\frac{M_P}{m_P}}\,
 \e^{L\tau}
 \left[
 \tau\varepsilon_1
 +H\varepsilon_0\psi_L(\tau)
 \right]
 <1-q
 \label{eq:C1-robustness-condition}
\end{equation}
ensures persistence.
\end{corollary}

\begin{pf}
The Lipschitz assumption gives $\omega_{DF}(r)\leq Hr$. Substitution into \eqref{eq:delta-vector-field-bound-general-C1} and the identity
\begin{equation}
 \int_0^\tau c_0(s)\,\dd s=\psi_L(\tau)
 \label{eq:c0-integral-psi}
\end{equation}
prove the claim.
\end{pf}

\begin{remark}
\label{rem:robustness-scope}
The base flow is fixed throughout this section. Hence the result concerns persistence of the fiber contraction and continuity of the graph and chain-recurrent graph over the same base chain structure. It is not a structural-stability theorem for perturbations of the entire skew product. The common invariant set $K$ is either assumed or verified by Proposition~\ref{prop:robust-invariance-K}.
\end{remark}

\subsection{Hausdorff continuity of chain recurrent sets}
\label{subsec:CR-robustness}

Let \(d_H\) denote the Hausdorff distance induced by the fixed product metric \(D\) in \eqref{eq:product-metric}.

\begin{corollary}
\label{cor:CR-robustness}
Under the assumptions of Theorem~\ref{thm:graph-robustness},
\begin{equation}
 d_H\bigl(\CR(\Pi),\CR(\widetilde\Pi)\bigr)
 \leq c_P\frac{\eta_\tau}{1-q},
 \label{eq:CR-Hausdorff-bound}
\end{equation}
where \(c_P:=\max\{1,m_P^{-1/2}\}\), as in Theorem~\ref{thm:exponential-tracking}.
\end{corollary}

\begin{pf}
Both cocycles satisfy finite-window contraction, so
\begin{equation}
 \CR(\Pi)=\{(\omega,h(\omega)):\omega\in\CR(\theta)\},
 \qquad
 \CR(\widetilde\Pi)
 =\{(\omega,\widetilde h(\omega)):\omega\in\CR(\theta)\}.
 \label{eq:two-recurrent-graphs}
\end{equation}
Pairing points with the same base coordinate and using \eqref{eq:graph-robustness-bound} proves the estimate.
\end{pf}

\subsection{Persistent disturbances and recurrent tubes}
\label{subsec:disturbance-tube}

We next quantify the effect of persistent perturbations on the transverse
distance from the invariant graph. A standard scalar comparison argument
yields both an asymptotic forward bound and a sharper uniform tube estimate
for complete trajectories and compact strictly invariant sets.

\begin{proposition}
	\label{prop:disturbance-tube}
	Let \(W\geq0\) be continuous along trajectories of a perturbed closed-loop system, and suppose that
	\begin{equation}
		D^+W(t)\leq-\lambda W(t)+\delta,
		\qquad \lambda>0,\quad\delta\geq0.
		\label{eq:ISS-graph-distance}
	\end{equation}
	Then every forward trajectory satisfies
	\begin{equation}
		W(t)
		\leq
		\e^{-\lambda(t-s)}W(s)
		+\frac{\delta}{\lambda}
		\bigl(1-\e^{-\lambda(t-s)}\bigr),
		\qquad t\geq s,
		\label{eq:ISS-graph-bound}
	\end{equation}
	and therefore
	\begin{equation}
		\limsup_{t\to\infty}W(t)
		\leq\frac{\delta}{\lambda}.
		\label{eq:forward-practical-bound}
	\end{equation}
	If a trajectory is complete and \(W\) is bounded along it, that is,
	\[
	\sup_{s\in\R}W(s)<\infty,
	\]
	then
	\begin{equation}
		W(t)\leq\frac{\delta}{\lambda},
		\qquad t\in\R.
		\label{eq:complete-recurrent-tube}
	\end{equation}
	In particular, this hypothesis is satisfied whenever the complete trajectory is contained
	in a compact set on which \(W\) is continuous. Consequently, every compact strictly
	invariant set of a flow is contained in the tube
	\[
	W\leq\frac{\delta}{\lambda}.
	\]
	The same conclusion holds for a compact strictly invariant set of a semiflow whenever
	each of its points lies on a complete trajectory contained in that set; this property holds,
	in particular, under the standard complete-orbit construction for compact strictly
	invariant semiflow sets.
\end{proposition}

\begin{pf}
	The scalar comparison lemma applied to \eqref{eq:ISS-graph-distance} gives
	\eqref{eq:ISS-graph-bound}, from which \eqref{eq:forward-practical-bound} follows
	immediately.
	
	Now let the trajectory be complete and assume that
	\[
	M_W:=\sup_{r\in\R}W(r)<\infty.
	\]
	Applying \eqref{eq:ISS-graph-bound} on the interval \([s,t]\), with \(s<t\), yields
	\[
	W(t)
	\leq
	\e^{-\lambda(t-s)}W(s)
	+\frac{\delta}{\lambda}
	\bigl(1-\e^{-\lambda(t-s)}\bigr)
	\leq
	M_W\e^{-\lambda(t-s)}
	+\frac{\delta}{\lambda}.
	\]
	Letting \(s\to-\infty\) gives \eqref{eq:complete-recurrent-tube}.
	
	Every point of a compact strictly invariant set of a flow belongs to a complete orbit
	contained in that set. Since \(W\) is continuous, it is bounded along every such orbit,
	and \eqref{eq:complete-recurrent-tube} applies. For a compact strictly invariant semiflow
	set, complete prehistories can be obtained by repeated backward selection followed by a
	compactness diagonal argument. The resulting complete trajectories remain in the compact
	set, so \(W\) is bounded along them. Applying \eqref{eq:complete-recurrent-tube} proves the
	final assertion.
\end{pf}

\section{Representative classes and examples}
\label{sec:examples}

This section illustrates the preceding theory on standard recurrent bases and on
examples where the averaged contraction mechanism is genuinely stronger than
pointwise contraction. We first record two immediate consequences for classical
periodic and quasiperiodic forcing and then turn to settings with several ergodic
components, transient expansion, and nontrivial chain order.

\subsection{Periodic and minimal recurrent bases}
\label{subsec:standard-recurrent-bases}

Two standard consequences may be stated briefly. If the base is a periodic circle
flow, then the invariant graph consists of a single periodic orbit and
\begin{equation}
	\CR(\Pi)=\Gamma_h.
	\label{eq:periodic-CR-graph}
\end{equation}
Hence no additional equilibrium, subharmonic orbit, or chaotic chain component can
arise inside the contracting fiber region. The entire chain-recurrent dynamics of
the skew product is carried by the periodic graph.

If the base is a minimal quasiperiodic torus flow, then the invariant graph is
minimal and again coincides with the complete chain-recurrent set. Under unique
ergodicity, the averaged moving-metric condition reduces to a single Haar integral,
as stated in Corollary~\ref{cor:uniquely-ergodic}. These are classical manifestations
of invariant-graph theory. The examples below are more informative because they
combine several ergodic components, transient fiber expansion, and a nontrivial
order structure among the chain components.

\subsection{A non-uniquely ergodic base with transient expansion}
\label{subsec:gradient-base}

Consider the gradient-like flow on \(\Omega=\T^1\) generated by
\begin{equation}
 \dot\vartheta=-\sin\vartheta.
 \label{eq:gradient-base}
\end{equation}
Its chain recurrent set consists of the stable equilibrium \(\vartheta=0\) and the unstable equilibrium \(\vartheta=\pi\):
\begin{equation}
 \CR(\theta)=\{0,\pi\}.
 \label{eq:gradient-base-CR}
\end{equation}
By Poincar\'e recurrence, every invariant probability measure is supported on these two equilibria. Hence
\begin{equation}
 \Mcal_\theta(\T^1)
 =\{p\delta_0+(1-p)\delta_\pi:0\leq p\leq1\}.
 \label{eq:gradient-invariant-measures}
\end{equation}

Consider the scalar plant and feedback
\begin{equation}
 \dot x=x-x^3+u+\cos\vartheta,
 \qquad
 u=-2\cos^2\vartheta\,x.
 \label{eq:gradient-controlled-fiber}
\end{equation}
The closed-loop equation is
\begin{equation}
 \dot x=a(\vartheta)x-x^3+\cos\vartheta,
 \qquad
 a(\vartheta):=1-2\cos^2\vartheta.
 \label{eq:gradient-closed-loop-fiber}
\end{equation}
Its fiber derivative satisfies
\begin{equation}
 \frac{\partial F}{\partial x}(\vartheta,x)
 =a(\vartheta)-3x^2\leq a(\vartheta).
 \label{eq:gradient-fiber-derivative}
\end{equation}
Although \(a(\vartheta)>0\) whenever \(\lvert\cos\vartheta\rvert<1/\sqrt2\), one has
\begin{equation}
 \int_{\T^1}a\,\dd\mu=-1
 \qquad
 \text{for every }\mu\in\Mcal_\theta(\T^1),
 \label{eq:gradient-all-measures-negative}
\end{equation}
because \(a(0)=a(\pi)=-1\). Thus, the instantaneous upper bound on fiber growth is positive along portions of every heteroclinic base orbit, while its average with respect to every invariant probability measure is strictly negative. In this example the common contraction window can also be estimated directly. Since
\begin{equation}
 a(\vartheta)=-1+2\sin^2\vartheta,
 \qquad
 \frac{\dd}{\dd t}\cos\vartheta(t)=\sin^2\vartheta(t),
 \label{eq:gradient-exact-integral-identities}
\end{equation}
one has, for every initial phase and every \(t\geq0\),
\begin{equation}
 \int_0^t a(\vartheta(s))\,\dd s
 =-t+2\bigl(\cos\vartheta(t)-\cos\vartheta(0)\bigr)\leq -t+4.
 \label{eq:gradient-explicit-window-bound}
\end{equation}
Therefore every \(\tau>4\) is a contraction window, with the explicit factor
\begin{equation}
 q\leq\e^{4-\tau}<1.
 \label{eq:gradient-explicit-q}
\end{equation}

Choose \(R>0\) such that
\begin{equation}
 R^3>R+1.
 \label{eq:gradient-invariant-interval}
\end{equation}
Then \(K=[-R,R]\) is positively invariant. Corollary~\ref{cor:control-certificate} applies with \(P=1\). Let \(\xi>0\) be the unique solution of
\begin{equation}
 \xi^3+\xi=1.
 \label{eq:xi-equation}
\end{equation}
At the two recurrent base points the lifted equilibria are \((0,\xi)\) and \((\pi,-\xi)\). Hence
\begin{equation}
 \CR(\Pi)=\{(0,\xi),(\pi,-\xi)\}.
 \label{eq:gradient-two-component-CR}
\end{equation}
This example simultaneously shows that unique ergodicity is unnecessary, transient expansion is allowed, and distinct chain components of the base remain distinct after feedback contraction.

The same example illustrates the attractor--repeller asymmetry. For the base attractor $A=\{0\}$, the standard base dual repeller is $A_\theta^*=\{\pi\}$. Hence
\begin{equation}
 (\iota(A))_\Pi^*
 =\{\pi\}\times K,
 \qquad
 \Core_\Pi^{\mathrm{str}}\bigl(\{\pi\}\times K\bigr)
 =\{(\pi,-\xi)\}.
 \label{eq:gradient-full-fiber-repeller}
\end{equation}
Thus every fiber point over $\pi$ belongs to the standard dual repeller, although only the graph equilibrium $(\pi,-\xi)$ belongs to its largest compact strictly invariant core.

\subsection{A two-dimensional controlled fiber with rotation and transient expansion}
\label{subsec:two-dimensional-example}

The preceding gradient-like base can also support a genuinely multidimensional controlled response. Let \(J\in\R^{2\times2}\) be the rotation generator
\begin{equation}
 J=\begin{pmatrix}0&-1\\1&0\end{pmatrix},
 \qquad J^\top=-J,
 \label{eq:rotation-generator}
\end{equation}
and consider the plant
\begin{equation}
 \dot x
 =x-\lVert x\rVert^2x+b(\vartheta)Jx+u+r(\vartheta),
 \qquad x\in\R^2,
 \label{eq:two-dimensional-plant}
\end{equation}
where \(b:\T^1\to\R\) and \(r:\T^1\to\R^2\) are continuous. Apply the feedback
\begin{equation}
 u=-2\cos^2\vartheta\,x.
 \label{eq:two-dimensional-feedback}
\end{equation}
The closed-loop fiber equation is
\begin{equation}
 \dot x
 =a(\vartheta)x+b(\vartheta)Jx
 -\lVert x\rVert^2x+r(\vartheta),
 \qquad
 a(\vartheta)=1-2\cos^2\vartheta.
 \label{eq:two-dimensional-closed-loop}
\end{equation}
Its Jacobian satisfies
\begin{equation}
 D_xF(\vartheta,x)
 =a(\vartheta)I+b(\vartheta)J
 -\lVert x\rVert^2I-2xx^\top.
 \label{eq:two-dimensional-jacobian}
\end{equation}
Since the rotational term is skew-symmetric,
\begin{equation}
 \frac{1}{2}\bigl(D_xF+D_xF^\top\bigr)
 =a(\vartheta)I-\lVert x\rVert^2I-2xx^\top
 \preceq a(\vartheta)I.
 \label{eq:two-dimensional-symmetric-bound}
\end{equation}
Thus \(P=I\) and \(\nu_P(\vartheta)\leq a(\vartheta)\). As in \eqref{eq:gradient-all-measures-negative}, every invariant measure of the gradient-like base satisfies
\begin{equation}
 \int_{\T^1}a(\vartheta)\,\dd\mu(\vartheta)=-1.
 \label{eq:two-dimensional-average-negative}
\end{equation}
Nevertheless, the symmetric part of the Jacobian is allowed to have a positive direction whenever \(a(\vartheta)>0\) and \(x\) is sufficiently small. The exact integral estimate \eqref{eq:gradient-explicit-window-bound} applies unchanged, so every \(\tau>4\) yields the explicit Euclidean contraction factor \(q\leq\e^{4-\tau}\).

Let \(R>0\) satisfy
\begin{equation}
 R^4>R^2+R\lVert r\rVert_\infty.
 \label{eq:two-dimensional-invariant-ball}
\end{equation}
On the boundary of the ball
\(K=\{x:\|x\|\leq R\}\), using \(x^\top Jx=0\), we obtain
\begin{equation}
	x^\top F(\vartheta,x)
	=
	a(\vartheta)R^2-R^4+x^\top r(\vartheta)
	\leq
	R^2-R^4+R\|r\|_\infty
	<0.
	\label{eq:vector-example-inward}
\end{equation}
Hence \(K\) is positively invariant. The averaged moving-metric theorem yields a unique invariant graph \(h:\T^1\to K\) and
\begin{equation}
 \CR(\Pi)=\{(0,p_0),(\pi,p_\pi)\},
 \label{eq:two-dimensional-chain-recurrent-set}
\end{equation}
where \(p_0=h(0)\) and \(p_\pi=h(\pi)\) are the unique solutions in \(K\) of
\begin{equation}
 -p_0+b(0)Jp_0-\lVert p_0\rVert^2p_0+r(0)=0,
 \qquad
 -p_\pi+b(\pi)Jp_\pi-\lVert p_\pi\rVert^2p_\pi+r(\pi)=0.
 \label{eq:recurrent-fiber-equilibria}
\end{equation}
Uniqueness follows from strict contraction of the autonomous fiber equations over the two recurrent base equilibria. The example shows that the chain classification is not restricted to scalar fibers: rotational dynamics, nonlinear radial damping, nonzero recurrent responses, and transient symmetric expansion are all compatible with the theory.

\subsection{Transient expansion over every forcing period}
\label{subsec:transient-expansion-example}

Consider the periodically forced controlled scalar equation
\begin{equation}
	\dot\vartheta=1,
	\qquad
	\dot x=(\alpha+\beta\cos\vartheta)x-x^3+u+\rho(\vartheta),
	\label{eq:scalar-controlled-example}
\end{equation}
where \(\rho\) is continuous and \(2\pi\)-periodic. Under the linear feedback
\begin{equation}
	u=-\kappa x,
	\label{eq:scalar-feedback}
\end{equation}
the fiber derivative satisfies
\begin{equation}
	\frac{\partial F}{\partial x}(\vartheta,x)
	=
	\alpha-\kappa+\beta\cos\vartheta-3x^2
	\leq
	a(\vartheta),
	\qquad
	a(\vartheta):=\alpha-\kappa+\beta\cos\vartheta.
	\label{eq:scalar-jacobian-bound}
\end{equation}
If
\begin{equation}
	\kappa>\alpha,
	\label{eq:scalar-average-condition}
\end{equation}
then the mean upper growth rate over one period is negative:
\begin{equation}
	\frac{1}{2\pi}
	\int_0^{2\pi}a(\vartheta)\,\dd\vartheta
	=
	\alpha-\kappa
	<0.
	\label{eq:scalar-average-negative}
\end{equation}
At the same time, the instantaneous upper growth rate may remain positive
on a nontrivial portion of each period. Indeed, whenever
\begin{equation}
	\lvert\beta\rvert>\kappa-\alpha,
	\label{eq:scalar-transient-expansion}
\end{equation}
the function \(a(\vartheta)\) changes sign and is positive on a nonempty
open subset of \([0,2\pi)\). Thus the averaged contraction condition is
compatible with recurrent intervals of transient fiber expansion.

Choose \(R>0\) sufficiently large that
\begin{equation}
 R^3>
 \bigl(\lvert\alpha-\kappa\rvert+\lvert\beta\rvert\bigr)R
 +\lVert\rho\rVert_\infty.
 \label{eq:scalar-invariant-interval-condition}
\end{equation}
Then \(K=[-R,R]\) is positively invariant. Corollary~\ref{cor:control-certificate} applies with \(P=1\). Therefore, there exists a unique periodic invariant graph \(h\), every trajectory in \(K\) converges exponentially to it, and
\begin{equation}
 \CR(\Pi)=\Gamma_h.
 \label{eq:scalar-example-CR}
\end{equation}
This example illustrates the essential gain over pointwise contraction: the closed loop is allowed to be instantaneously expanding, but the complete recurrent response remains unique.

\subsection{Translation hulls of recurrent forcing}
\label{subsec:hulls}

Let \(p:\R\to\R^m\) be a bounded uniformly continuous forcing, and let
\(\Omega=H(p)\) be the closure of its translates in the compact-open topology.
The Bebutov translation
\begin{equation}
	(\theta_t\omega)(s)=\omega(s+t)
	\label{eq:Bebutov-translation}
\end{equation}
defines a compact base flow whenever the hull is compact. A nonautonomous controlled equation
\begin{equation}
 \dot x=f(p(t),x)+B(p(t),x)u
 \label{eq:single-forcing-equation}
\end{equation}
then embeds into the family over \(H(p)\). If the averaged metric condition holds on the entire hull, then the recurrent responses corresponding to recurrent elements of the hull lie on a single
continuous invariant graph. This formulation includes periodic, quasiperiodic, almost periodic, and more general translation-compact forcing; see \cite{KloedenRasmussen2011,Sell1971,ShenYi1998}.

\section{Discussion and further directions}
\label{sec:discussion}

Several extensions of the present framework are natural. The compact convex fiber
\(K\subset\R^n\) may be replaced by a compact bundle with complete metric fibers,
provided that the equivariant section, the fiber distance, and a common finite-window
contraction estimate depend continuously on the base point. Convexity is used only in
the derivative-based nonlinear estimate. The topological chain-reduction mechanism
itself requires only a factor map, an equivariant section, and uniform asymptotic
tracking. Likewise, a state-dependent Riemannian or Finsler contraction metric may
replace the base-dependent quadratic metric whenever its induced geodesic distance
satisfies a uniform finite-window estimate. A common invariant fiber set may also be
replaced by a compact absorbing family if asymptotic compactness yields a suitable
compact ambient phase space.

The differential criterion developed here uses negativity of the maximal
invariant-measure exponent to obtain a deterministic contraction window that is uniform
over the entire base. If contraction holds only almost surely with respect to selected
invariant measures, the attracting object may be merely measurable and the ambient
topological chain reduction need not survive. This distinction separates the uniform
source--target classification obtained here from measure-dependent synchronization
mechanisms.

From the control viewpoint, feedback synthesis may target either the maximal fiber
exponent \(\chi_*\) or its computable moving-metric upper bound \(\beta_P\). Both
quantities incorporate the invariant-measure structure of the exosystem and allow
transiently expansive regions. The examples demonstrate that a common contraction
window and its contraction factor can sometimes be estimated explicitly even when the
instantaneous logarithmic norm changes sign. This suggests optimization-based feedback
designs aimed at maximizing the averaged contraction margin without requiring
pointwise contraction at every base state.

\section{Conclusion}
\label{sec:conclusion}

We established an abstract chain-reduction theorem for semiflows admitting a factor
map, an equivariant section, and uniform asymptotic tracking. The central result,
Theorem~\ref{thm:abstract-chain-reduction}, gives an exact ambient source--target
classification: for every source point, including points outside the attracting
section, the complete set of chain-attainable targets is the section lift of the
chain-attainable set of its factor coordinate. Corollary~\ref{cor:abstract-recurrent-reduction}
consequently identifies the chain-recurrent core, its components, and their directed
order.
For contracting controlled skew-product semiflows,
Corollary~\ref{cor:controlled-chain-reduction} yields
\[
z\rightsquigarrow_\Pi w
\quad\Longleftrightarrow\quad
w=\iota(\eta)
\ \text{for some }\eta\in\Omega
\text{ such that }
\pi_\Omega(z)\rightsquigarrow_\theta\eta.
\]
Theorems~\ref{thm:chain-recurrence-identity} and
\ref{thm:chain-component-correspondence} then recover the full recurrent structure of
the skew product from the base dynamics. Theorem~\ref{thm:complete-lyapunov-lift}
complements this reduction by constructing an explicit complete Lyapunov function
through a weighted-supremum transverse gauge, with exponential semigroup decay and
exact preservation of recurrent values.

The invariant-set classification reveals a sharp dual asymmetry. By
Theorems~\ref{thm:compact-invariant-correspondence} and
\ref{thm:attractor-correspondence}, compact strictly invariant sets and attractors
collapse to graph lifts, and their basins are full inverse images of the corresponding
base basins. In contrast, Theorem~\ref{thm:attractor-repeller-correspondence} shows that
standard dual repellers retain entire fibers over the base repeller, while only their
largest compact strictly invariant cores collapse to the graph. Thus the source--target
asymmetry of the ambient chain relation and the attractor--repeller asymmetry are two
manifestations of the same contracting-extension geometry.

Finally, Theorems~\ref{thm:subadditive-main} and
\ref{thm:averaged-contraction-main} make the topological reduction verifiable for
controlled nonautonomous differential equations. A negative maximal
invariant-measure fiber exponent yields a common contraction window, while the
moving-metric logarithmic-norm criterion permits transient expansion and non-unique
ergodicity. The perturbation results provide quantitative persistence under
finite-window and vector-field perturbations, including general \(C^1\) estimates and
an explicit \(C^{1,1}\) specialization.

\end{document}